\documentclass[10pt]{article}
\usepackage{amsmath}
\usepackage{amssymb}
\usepackage{amsfonts}

\newtheorem{theorem}{Theorem}[section]

\newtheorem{definition}[theorem]{Definition}

\newtheorem{lemma}[theorem]{Lemma}

\newtheorem{remark}[theorem]{Remark}

\newenvironment{proof}[1][Proof]{\noindent\textbf{#1.} }{\ \rule{0.5em}{0.5em}}
\begin{document}

\title{}
\author{}
\date{}
\maketitle

\renewcommand{\theequation}{\arabic{section}.\arabic{equation}}

\bigskip

\begin{center}
{\Large The} ${\LARGE H}^{\infty }${\Large -control problem for parabolic
systems with singular Hardy potentials}

\bigskip

Gabriela Marinoschi

\medskip

\textquotedblleft Gheorghe Mihoc-Caius Iacob\textquotedblright\ Institute of
Mathematical Statistics and

Applied Mathematics of the Romanian Academy,

Calea 13 Septembrie 13, Bucharest, Romania

\bigskip
\end{center}

\noindent Abstract. We solve the $H^{\infty }$-control problem with state
feedback for infinite dimensional boundary control systems of parabolic type
with distributed disturbances and apply the results to equations with Hardy
potentials with the singularity inside or on the boundary, in the cases of a
distributed control and of a boundary control.

\medskip


\medskip


\section{Introduction\label{Intro}}

\setcounter{equation}{0}

The $H^{\infty }$-control is a technique used in control theory to design
robust stabilizing feedback controllers that force a system to achieve
stability with a prescribed performance even if the system output may be
corrupted by perturbations. This method involves a transfer function which
incorporates the effects of the input perturbations towards the output
observation. The aim is to determine the optimal feedback controller which
minimizes the effect of these perturbations on the output, by ensuring that
the $L^{2}$-norm of the transfer function is smaller that the $L^{2}$-norm
of the perturbation with a certain prescribed bound. This turns out in
finding a suboptimal control solution constructed by means of a mathematical
optimization problem. The formal $H^{\infty }$-control theory was initiated
by Zames in \cite{Zames}, as an optimization problem with an operator norm,
in particular, the $H^{\infty }$-norm. State space formulations were
initially developed in \cite{Glover} and \cite{Packard-Doyle} and continued
later by the formulation of the necessary and sufficient conditions for the
existence of an admissible controller in terms of solutions of algebraic
Riccati equations. The state-space approach for linear infinite-dimensional $%
H^{\infty }$-control problems was developed in further works and we cite
here e.g., \cite{VB-H-92}, \cite{VB-H-95-hyp}, \cite{VB-H-95}, \cite{Bas-Ber}%
, \cite{Opmeer-Staffans-2010}, \cite{Opmeer-Staffans-2019}, \cite{Pri-Tow}, 
\cite{vanK-93}, \cite{vanK-al-93}, \cite{vanS-91}, \cite{vaS-92}, \cite{VB-S}%
, the last for Navier-Stokes equations.

In this paper we discuss the $H^{\infty }$-control problem for linear
infinite dimensional systems of parabolic type and give applications for
equations with singular Hardy potentials, of the type $\frac{\lambda }{%
|x|^{2}},$ which as far as we know is a novel approach. Following the papers 
\cite{VB-H-92}, \cite{VB-H-95-hyp}, \cite{VB-H-95}, where the $H^{\infty }$%
-control abstract problem was solved with assumptions proper for the
hyperbolic case, we prove here a main result stating the formulation of the $%
H^{\infty }$-control problem in the parabolic case, relying on appropriate
assumptions for parabolic operators. This is further applied to three
parabolic control systems with Hardy potentials and with distributed or
boundary controls. There is an extensive literature on Hardy-type
inequalities with the singularity located inside the domain or on the
boundary, focusing also on controllability studies (see e.g., \cite{Ervedoza}%
, \cite{Fragnelli-Mugnai}, \cite{Vancostenoble-Zuazua}). Besides the high
mathematical interest in such singular equations revealed in the past
decades, a parabolic operator with a Hardy potential term describes a
non-standard growth condition which may affect the behavior of the solutions
to diffusive physical models, as for example of heat transfer or diffusion
of contaminants in fluids. Also, it may represent an equivalent formulation
of a system of two equations in which a state in one equation is represented
as a fundamental solution by the other one. Operators with other similar
potentials can arise for example in quantum mechanics, \cite{Baras-Goldstein}
or in combustion theory, \cite{Bebernes}, \cite{Fragnelli-Mugnai}. Linear
parabolic equations with Hardy potentials have been studied in connection
with stationary Schr\"{o}dinger equations $-\Delta y+V(x)y+E(x)y=f$ with the
singular potential $V\in L^{\infty }(\Omega \backslash x_{0})$ arising from
the uncertainty principle. The robust stabilization of the corresponding
dynamic control system $y_{t}-\Delta y+V(x)y+E(x)y=B_{1}w+B_{2}u$, via the $%
H^{\infty }$-control method, with the control $u$ and the exogenous
perturbation $w$ has direct implication for the equilibrium solution to the
above Schr\"{o}dinger equation. The content of the paper is briefly
described below.

In Section \ref{Prel} we present the mathematical formulation of the $%
H^{\infty }$-control problem. In Section \ref{Main}, after specifying the
work hypotheses we provide the main result stating the existence of the
feedback controller determined via a Riccati equation. In Sections \ref%
{Distributed} and \ref{Boundary} there are given applications for parabolic
equations in the $N$-dimensional case with a distributed control and a
boundary control, respectively, and with Hardy potentials with interior
singularity, while in Section \ref{1D} it is treated the $1$-$D$ case with a
boundary singular Hardy potential.

\section{Problem presentation and preliminaries\label{Prel}}

\setcounter{equation}{0}

In this section we briefly explain the state-space approach of the $%
H^{\infty }$-control problem for the linear system%
\begin{eqnarray}
&&y^{\prime }(t)=Ay(t)+B_{1}w(t)+B_{2}u(t),\mbox{ \ }t\in \mathbb{R}%
_{+}:=(0,+\infty )  \label{1} \\
&&z(t)=C_{1}y(t)+D_{1}u(t),\mbox{ }t\in \mathbb{R}_{+},  \label{2} \\
&&y(0)=y_{0},  \label{1'}
\end{eqnarray}%
where $A,$ $B_{1},$ $B_{2},$ $C_{1},$ $D_{1}$ are linear operators
satisfying hypotheses that will be immediately specified. Here, $y$ is the
system state, $u$ is the control input, $w$ is an exogenous input, or an
unknown perturbation and $z$ is the performance output.

At this point we put down a few notation, definitions and results necessary
for explaining the problem. Let $X$ be a real Hilbert space with the scalar
product and norm denoted by $(\cdot ,\cdot )_{X}$ and $\left\Vert \cdot
\right\Vert _{X},$ respectively and $X^{\prime }$ is its dual. The symbol $%
\left\langle \cdot ,\cdot \right\rangle _{X^{\prime },X}$ is the pairing
between $X^{\prime }$ and $X.$ Let $A$ be a linear closed operator on $X$
with the domain $D(A):=\{y\in X;$ $Ay\in X\}$ dense in $X.$ By $A^{\ast }$
we denote the the adjoint of $A.$ If $Y$ is another Hilbert space, $L(X,Y)$
represent the space of all linear continuous operators from $X$ to $Y.$

Let $H,$ $U,$ $W,$ $Z$ be real Hilbert spaces identified with their duals.
For the beginning we assume:

\begin{itemize}
\item[$(i_{1})$] $A$ is the infinitesimal generator of an analytic $C_{0}$%
-semigroup $e^{At}$ on the Hilbert space $H,$ $e^{At}$ is compact for $t>0,$
and 
\begin{equation}
B_{1}\in L(W,H),\mbox{ }B_{2}\in L(U,(D(A^{\ast }))^{\prime }),\mbox{ }%
C_{1}\in L(H,Z),\mbox{ }D_{1}\in L(U,Z).  \label{1i}
\end{equation}%
Here, $(D(A^{\ast }))^{\prime }$ is the dual of the domain of $A^{\ast }$,
where $D(A^{\ast })$ is organized as a Hilbert space with the scalar product 
$(y_{1},y_{2})_{D(A^{\ast })}=(A^{\ast }y_{1},A^{\ast
}y_{2})_{H}+(y_{1},y_{2})_{H}$ for $y_{1},$ $y_{2}\in D(A^{\ast }).$
\end{itemize}

\noindent We note that the space $(D(A^{\ast }))^{\prime }$ is the
completion of $H$ in the norm $\left\vert \left\Vert y\right\Vert
\right\vert =\left\Vert (A-\lambda _{0}I)^{-1}y\right\Vert _{H},$ $\lambda
_{0}\in \rho (A).$ Also, we define the extension of the operator $A$ from $H$
to $(D(A^{\ast }))^{\prime },$ denoted for convenience still by $A,$ by 
\begin{equation}
\left\langle Ay,\psi \right\rangle _{(D(A^{\ast }))^{\prime },D(A)}=\left(
y,A^{\ast }\psi \right) _{H},\mbox{ for }y\in H,\mbox{ }\psi \in D(A^{\ast
}).  \label{401}
\end{equation}%
We shall work with both operators and if not seen clearly from the context
which operator is used, we shall specify this.

\noindent Let us consider the uncontrolled system $y^{\prime }(t)=Ay(t),$ $%
t\in \mathbb{R}_{+},$ $y(0)=y_{0},$ with $A$ the infinitesimal generator of
a $C_{0}$-semigroup on $H.$

\begin{definition}
\textrm{The operator $A$ generates an exponentially stable semigroup $e^{At}$
if 
\begin{equation}
\left\Vert e^{At}\right\Vert _{L(H,H)}\leq Ce^{-\alpha t},\mbox{ for all }%
t\geq 0,  \label{4}
\end{equation}%
where $\alpha $ and $C$ are positive constants. }
\end{definition}

\noindent Relation (\ref{4}) still reads%
\begin{equation}
\left\Vert e^{At}y\right\Vert _{H}\leq Ce^{-\alpha t}\left\Vert y\right\Vert
_{H},\mbox{ for all }y\in H\mbox{ and all }t\geq 0.  \label{5}
\end{equation}

\noindent Moreover, a result of Datko (see \cite{Datko}) asserts that
relation (\ref{5}) is equivalent to 
\begin{equation}
\int_{0}^{\infty }\left\Vert y(t)\right\Vert _{H}^{2}dt<\infty .  \label{6}
\end{equation}

\begin{definition}
\textrm{The pair $(A,C_{1})$ in system (\ref{1})-(\ref{2}) is exponentially
detectable if there exists $K\in L(Z,H)$ such that $A+KC_{1}$ generates an
exponentially stable semigroup. }
\end{definition}

In order to state our $H^{\infty }$-control problem, we recall some issues
about such a problem. Assume that under certain conditions system (\ref{1})-(%
\ref{1'}) has a mild solution $y\in C([0,T];H)$ for all $T>0$ and $u$ can be
represented as a feedback controller $u=Fy,$ where generally $F:U\rightarrow
H$ is a linear closed and densely defined operator. Then, the solution $%
(y(t),z(t))$ becomes dependent only on $w(t)$ and reads 
\begin{eqnarray}
&&y(t)=e^{(A+B_{2}F)t}y_{0}+\displaystyle%
\int_{0}^{t}e^{(A+B_{2}F)(t-s)}B_{1}w(s)ds,\ t\in \lbrack 0,\infty ),
\label{7} \\
&&z(t)=(C_{1}+D_{1}F)e^{(A+B_{2}F)t}y_{0}+(C_{1}+D_{1}F)\displaystyle%
\int_{0}^{t}e^{(A+B_{2}F)(t-s)}B_{1}w(s)ds.  \label{8}
\end{eqnarray}%
The latter equation can be still written 
\begin{equation}
z(t)=f_{0}(t)+(G_{F}w)(t),\mbox{ }t\geq 0  \label{8-1}
\end{equation}%
where $f_{0}(t)=(C_{1}+D_{1}F)e^{(A+B_{2}F)t}y_{0}\in Z,$ $t\geq 0,$ and $%
G_{F}:L^{2}(\mathbb{R}_{+},W)\rightarrow L^{2}(\mathbb{R}_{+},Z),$ defined by%
\begin{equation}
(G_{F}w)(t)=(C_{1}+D_{1}F)\int_{0}^{t}e^{(A+B_{2}F)(t-s)}B_{1}w(s)ds\in Z,%
\mbox{ }t\geq 0,  \label{8-2}
\end{equation}%
shows the transfer of the influence of the perturbation input $w$ to the
output. Roughly speaking, the $H^{\infty }$-control problem means to find a
feedback controller which stabilizes exponentially the system (with $%
y_{0}=0) $, with a certain specified performance for the output $G_{F}w$,
depending on a given constant $\gamma .$ Such a feedback control $F$ is
called a suboptimal solution and the $H^{\infty }$ problem can be formulated
as follows: given $\gamma >0,$ find the feedback control $F$ which
exponentially stabilizes system (\ref{1})-(\ref{2}) such that $\left\Vert
G_{F}\right\Vert _{L(L^{2}(\mathbb{R}_{+},W),L^{2}(\mathbb{R}%
_{+},Z))}<\gamma .$

To be more precise in what concerns the relation with the Hardy space $%
H^{\infty }$, we briefly recall a well-known result property of
vector-valued Hardy classes (see e.g., \cite{Staffans}, \cite{Staffans-Weiss}%
, \cite{Curtain-Zwart}, Theorem A6.26). The space $H^{\infty }$ is defined
as the vector space of bounded holomorphic functions on the right half
plane, $\mathbb{C}_{+}=\{z\in \mathbb{C};$ $\mathit{Re}\,z>0\}$, with the
norm $\left\Vert f\right\Vert _{H^{\infty }}=\sup_{\left\vert z\right\vert
<1}\left\vert f(z)\right\vert .$ Let us take the Laplace transform in system
(\ref{1})-(\ref{2}) and get 
\begin{equation}
\widehat{z}(\zeta )=C_{1}(\zeta I-A-B_{2}F)^{-1}y_{0}+\widehat{G_{F}}(\zeta )%
\widehat{w}(\zeta ).  \label{9-0}
\end{equation}%
The function $\widehat{G_{F}}:\mathbb{C}_{+}\rightarrow L(W,Z),$ 
\begin{equation}
\widehat{G_{F}}(\zeta )=(C_{1}+D_{1}F)(\zeta I+A+B_{2}F)^{-1}B_{1}
\label{9-1}
\end{equation}%
is the transfer function in the frequency domain, giving a relationship
between the input and output of the system. It plays an important role in
control theory by providing an insight in how disturbances in the system can
affect the output. The results in the papers cited before express the fact
that the $L^{2}$-operator norm of the gain in the time domain is equal to
the Hardy $H^{\infty }(L(W,Z))$-norm of the transfer operator in the
frequency domain, i.e.,

\begin{equation}
\left\Vert G_{F}\right\Vert _{L(L^{2}(\mathbb{R}_{+},W),L^{2}(\mathbb{R}%
_{+},Z))}:=\sup_{w\in L^{2}(\mathbb{R}_{+},W)}\frac{\left\Vert
G_{F}w\right\Vert _{L^{2}(\mathbb{R}_{+},Z)}}{\left\Vert w\right\Vert
_{L^{2}(\mathbb{R}_{+},W)}}=\sup_{\zeta \in \mathbb{C}_{+}}\left\Vert 
\widehat{G_{F}}(\zeta )\right\Vert _{L(W,Z)}=:\left\Vert \widehat{G_{F}}%
\right\Vert _{H^{\infty }}<\gamma .  \label{9-2}
\end{equation}

\textit{Notation and some necessary results.} We end this section by
recalling some other notation and results necessary in the paper. We denote
by $H^{m}(\Omega )$ the Sobolev spaces $W^{2,m}(\Omega ),$ for $m\geq 1$ and
by $H_{0}^{1}(\Omega )$ the space $\{y\in H^{1}(\Omega );$ $tr(y)=0$ on $%
\Gamma \},$ where $tr(y)$ is the trace operator of $y$ on $\Gamma :=\partial
\Omega .$ Moreover, $H^{-1}(\Omega )$ denotes the dual of $H_{0}^{1}(\Omega
).$ Given a Banach space $X$ and $T\in (0,\infty ]$ we define by $%
L^{p}(0,T;X)$ the space of $L^{p}$ $X$-valued functions on $(0,T),$ $p\in
\lbrack 1,\infty ],$ by $C([0,T];X)$ the space of continuous $X$-valued
functions on $(0,T)$ and $W^{1,p}(0,T;X)=\{u\in L^{p}(0,T;X);$ $du/dt\in
L^{p}(0,T;X)\}.$

Let $L:D(L)\subset H\rightarrow H$ be a linear operator defined on the
Hilbert space $H.$ We say that $L$ is $m$-accretive if $L$ is accretive,
meaning that $(Ly,y)_{H}\geq 0,$ $\forall y\in D(L),$ and if $R(I+L)=H,$
where $R$ is the range. The operator $L$ is quasi $m$-accretive or $\omega $-%
$m$-accretive if $\omega I+L$ is $m$-accretive for some $\omega >0.$

\textit{Hardy inequalities.} Let $N>3$ and let $\Omega $ be an open bounded
subset of $\mathbb{R}^{N},$\textit{\ }with $0\in \Omega .$ Then we have 
\begin{equation}
\int_{\Omega }\left\vert \nabla y(x)\right\vert ^{2}dx\geq H_{N}\int_{\Omega
}\frac{\left\vert y(x)\right\vert ^{2}}{\left\vert x\right\vert ^{2}}dx,%
\mbox{ for all }y\in H_{0}^{1}(\Omega ),  \label{HN}
\end{equation}%
where $H_{N}=\frac{(N-2)^{2}}{4}$ is optimal (see \cite{Brezis-Vazquez}, p.
452, Theorem 4.1).

Let $\Omega =(0,1).$ Then we have 
\begin{equation}
\int_{0}^{1}\left\vert y^{\prime }(x)\right\vert ^{2}dx\geq \frac{1}{4}%
\int_{0}^{1}\frac{y(x)}{\left\vert x\right\vert ^{2}}dx,\mbox{ }\forall y\in
H^{1}(0,1),\mbox{ }y(0)=0,  \label{HN0}
\end{equation}%
see \cite{Brezis-Marcus}, p. 217, or Lemma A.1, p. 234.

We recall the Young's inequality for convolutions $(f\ast
g)(t)=\int_{0}^{\infty }f(t-\tau )g(\tau )d\tau ,$ 
\begin{equation}
\left\Vert f\ast g\right\Vert _{L^{r}(0,\infty )}\leq \left\Vert
f\right\Vert _{L^{p}(0,\infty )}\left\Vert g\right\Vert _{L^{q}(0,\infty )},%
\mbox{ where }\frac{1}{p}+\frac{1}{q}=1+\frac{1}{r},\mbox{ }1\leq p,\mbox{ }%
q,\mbox{ }r\leq \infty .  \label{400}
\end{equation}

For simplicity, where there is no risk of confusion, the $L^{p}(\Omega )$%
-norm will be denoted by $\left\Vert \cdot \right\Vert _{p},$ $p\in \lbrack
1,\infty ],$ instead of $\left\Vert \cdot \right\Vert _{L^{p}(\Omega )}.$ We
set $\mathbb{R}=(-\infty ,\infty )$ and $\mathbb{R}_{+}=(0,\infty ).$ Also, $%
\left\vert \cdot \right\vert $ will represent the Euclidian norm in $\mathbb{%
R}^{N},$ for any $N=1,2,...,$ accordingly. In the further calculations $C,$ $%
C_{1},...,C_{N},$ $C_{T}$ denote positive constants (which may change from
line to line), $C_{N}$ depending on $N,$ via $\lambda <H_{N}$ and $C_{T}$
depending on $T.$

\section{The main result\label{Main}}

\setcounter{equation}{0}

Besides $(i_{1})$ we assume the following hypotheses:

\begin{itemize}
\item[$(i_{2})$] the next relation takes place: 
\begin{equation}
\left\Vert B_{2}^{\ast }e^{A^{\ast }t}\right\Vert _{L(H,U)}\in L^{1}(0,T),%
\mbox{ for all }T>0,  \label{12}
\end{equation}

\item[$(i_{3})$] the pair $(A,C_{1})$ is exponentially detectable (that is
there exists $K\in L(Z,H)$ such that $A+KC_{1}$ generates an exponentially
stable semigroup) and 
\begin{equation}
\int_{0}^{\infty }\left\Vert B_{2}^{\ast }e^{(A^{\ast }+C_{1}^{\ast }K^{\ast
})t}y\right\Vert _{U}dt\leq C\left\Vert y\right\Vert _{H},\mbox{ for all }%
y\in H,  \label{12-0}
\end{equation}

\item[$(i_{4})$] $\left\Vert D_{1}^{\ast }D_{1}u\right\Vert _{U^{\ast
}}=\left\Vert u\right\Vert _{U}$ and $D_{1}^{\ast }C_{1}=0.$
\end{itemize}

\noindent Let us comment a little these hypotheses. The $L^{1}$%
-admissibility hypothesis of the observation operator $B_{2}^{\ast }$ in $%
(i_{2})$ is made in order to ensure the existence of a mild solution to (\ref%
{1}) in $L^{2}(0,T;H)$ for every $T>0,$ with initial condition $y_{0}$ and
inputs $u\in L^{2}(0,T;U)$ and $w\in L^{2}(0,T;W)$. In an ideal situation
when $B_{2}\in L(U,H),$ eqs. (\ref{1})-(\ref{1'}) have a unique mild
solution $y\in C([0,T];H),$ for every $T>0,$ given by%
\begin{equation}
y(t)=e^{At}y_{0}+\int_{0}^{t}e^{A(t-s)}B_{1}w(s)ds+%
\int_{0}^{t}e^{A(t-s)}B_{2}u(s)ds,\mbox{ }t\in \lbrack 0,\infty ).
\label{13}
\end{equation}%
But generally, $B_{2}$ may be not continuous from $U$ to $H,$ in some
situations its range being in a larger abstract space, indicated before to
be $(D(A^{\ast }))^{\prime }.$ The unique solution to (\ref{1})-(\ref{1'})
is in this case in $C([0,\infty );(D(A^{\ast }))^{\prime }).$ Consequently,
the previous formula should be written in a weak sense, that is for all $%
t\geq 0,$ we have%
\begin{equation}
(y(t),\varphi )_{H}=(e^{At}y_{0},\varphi )_{H}+\int_{0}^{t}\left(
e^{A(t-s)}(B_{1}w(s),\varphi )_{H}+(u(s),B_{2}^{\ast }e^{A^{\ast
}(t-s)}\varphi )_{U}\right) ds,\mbox{ }\forall \varphi \in H,\mbox{ }%
y_{0}\in H.  \label{13-1}
\end{equation}%
Assumption $(i_{2})$ ensures that $y\in L^{2}(0,T;H),$ and this follows by
proving that $\int_{0}^{T}(y(t),\varphi (t))_{H}dt<C\left\Vert \varphi
\right\Vert _{L^{2}(0,T;H)},$ for $\varphi \in L^{2}(0,T;H).$ Indeed, this
is clearly seen for the first two terms in (\ref{13-1}), since $B_{1}w\in
L^{2}(\mathbb{R}_{+};H)$. For the last term we calculate%
\begin{eqnarray}
&&\int_{0}^{T}\int_{0}^{t}\left( u(s),B_{2}^{\ast }e^{A^{\ast }(t-s)}\varphi
(t)\right) _{U}dsdt=\int_{0}^{T}\int_{s}^{T}\left( u(s),B_{2}^{\ast
}e^{A^{\ast }(t-s)}\varphi (t)\right) _{U}dtds  \label{16-0} \\
&\leq &\left( \int_{0}^{T}\left\Vert u(s)\right\Vert _{U}^{2}ds\right)
^{1/2}\left( \int_{0}^{T}\left\Vert \int_{0}^{T}B_{2}^{\ast }e^{A^{\ast
}(t-s)}\varphi (t)dt\right\Vert _{U}^{2}ds\right) ^{1/2}  \notag \\
&\leq &\left\Vert u\right\Vert _{L^{2}(0,T;U)}\left\{ \left(
\int_{0}^{T}\left\Vert B_{2}^{\ast }e^{A^{\ast }t}\right\Vert
_{L(H,U)}dt\right) \left( \int_{0}^{T}\left\Vert \varphi (t)\right\Vert
_{H}^{2}dt\right) ^{1/2}\right\}   \notag \\
&\leq &\left\Vert u\right\Vert _{L^{2}(0,T);U)}\left( \int_{0}^{T}\left\Vert
B_{2}^{\ast }e^{A^{\ast }t}\right\Vert _{L(H,U)}dt\right) \left\Vert \varphi
\right\Vert _{L^{2}(0,T;H)}\leq C\left\Vert \varphi \right\Vert
_{L^{2}(0,T;H)},  \notag
\end{eqnarray}%
where we used $(i_{2})$ and the Young's inequality for convolution (\ref{400}%
) with $p=1,$ $q=r=2.$ Then, it follows that $y\in L^{2}(0,T;H)$ and the
last term in (\ref{13}) is in $H$.

Regarding (\ref{12-0}) we mention that the corresponding result related to $%
L^{2}$ instead of $L^{1}$ is a particular case of Theorem 5.4.2 in \cite%
{Tucsnak-Weiss}, so that we expect that (\ref{12}) and the detectability
hypothesis imply (\ref{12-0}), at least in some cases. However, we keep here
relation (\ref{12-0}) as a hypothesis and check it in the applications, by
different proofs according the case. In applications, the first relation in
hypothesis $(i_{4})$ may be weaken to $D_{1}^{\ast }D_{1}\geq \epsilon I$
(see e.g., \cite{VB-S}). However, for certain choices of operators $D_{1}$
and $C_{1},$ relations $(i_{4})$ may be proved as they are.

Theorem 3.1 below is the main result concerning the $H^{\infty }$-control
problem under hypotheses $(i_{1})$-$(i_{4})$ and it gives a representation
for the feedback operator $F$ which is a suboptimal solution to our $%
H^{\infty }$-control problem.

This theorem was proved, under some appropriate hypotheses for the
hyperbolic case in \cite{VB-H-92} and \cite{VB-H-95-hyp}. Actually, instead
of (\ref{12}) there it was used the $L^{2}$-admissibility condition 
\begin{equation}
\int_{0}^{T}\left\Vert B_{2}^{\ast }e^{A^{\ast }t}y\right\Vert
_{U}^{2}dt\leq C_{T}\left\Vert y\right\Vert _{H}^{2},\mbox{ for every }y\in H%
\mbox{ and }T>0.  \label{14}
\end{equation}%
For the treatment of specific parabolic problems intended to be achieved in
the paper, we have in mind to adapt that approach to the case covered by
assumptions $(i_{1})-(i_{4})$ to obtain the following main result.

\begin{theorem}
\label{Th-main}Let hypotheses $(i_{1})-(i_{4})$ hold and let $\gamma >0.$
Assume that there exists $F\in L(H,U)$ such that $A+B_{2}F$ generates an
analytic exponentially stable $C_{0}$-semigroup on $H$ and 
\begin{equation}
\left\Vert G_{F}\right\Vert _{L(L^{2}(\mathbb{R}_{+};W),L^{2}(\mathbb{R}%
_{+};Z))}<\gamma .  \label{14-1}
\end{equation}%
Then, there exists a Hilbert space $\mathcal{X}\subset H$ with dense and
continuous injection and an operator 
\begin{equation}
P\in L(H,H)\cap L(\mathcal{X},D(A^{\ast })),\mbox{ }P=P^{\ast }\geq 0,
\label{14-2}
\end{equation}%
which satisfies the algebraic Riccati equation 
\begin{equation}
A^{\ast }Py+P(A-B_{2}B_{2}^{\ast }P+\gamma ^{-2}B_{1}B_{1}^{\ast
}P)y+C_{1}^{\ast }C_{1}y=0,\mbox{ }\forall y\in \mathcal{X},  \label{15}
\end{equation}%
where $B_{2}^{\ast }P\in L(\mathcal{X},U)$ and the operators 
\begin{equation}
\Lambda _{P}:=A-B_{2}B_{2}^{\ast }P+\gamma ^{-2}B_{1}B_{1}^{\ast }P,\mbox{ }%
\Lambda _{P}^{1}:=A-B_{2}B_{2}^{\ast }P  \label{15-prim}
\end{equation}%
with the domain $\mathcal{X}$ generate exponentially stable semigroups on $H$%
. Moreover, the feedback control 
\begin{equation}
\widetilde{F}=-B_{2}^{\ast }P  \label{16}
\end{equation}%
solves the $H^{\infty }$-problem, that is $\left\Vert G_{\widetilde{F}%
}\right\Vert _{L(L^{2}(\mathbb{R}_{+};W),L^{2}(\mathbb{R}_{+};Z))}<\gamma .$

Conversely, assume that there exists a solution $P$ to equation (\ref{15})
with the properties (\ref{14-2}) and such that the corresponding operators $%
\Lambda _{P}$ and $\Lambda _{P}^{1}$ generate exponentially stable
semigroups on $H.$ Then, the feedback operator $\widetilde{F}=-B_{2}^{\ast
}P $ solves the $H^{\infty }$-problem (\ref{14-1}).
\end{theorem}

The space $\mathcal{X}$ will be defined in the theorem proof before Lemma %
\ref{operators}, in (\ref{3-43-2}). Moreover, we shall show in Lemma \ref%
{operators} that if the operator $\Lambda _{P}$ with the domain $D(\Lambda
_{P})=\{y\in H;$ $\Lambda _{P}y=(A-B_{2}B_{2}^{\ast }P+\gamma
^{-2}B_{1}B_{1}^{\ast }P)y\in H\}$ is closed, then $\mathcal{X}=D(\Lambda
_{P}).$ This will happen in all examples given the next sections.

\medskip

\noindent \textbf{Proof }of Theorem \ref{Th-main}\textbf{. }We assume first
that there exists a solution $F\in L(H,U)$ to the $H^{\infty }$-control
problem such that $A_{F}:=A+B_{2}F$ generates an analytic exponentially
stable $C_{0}$-semigroup and (\ref{14-1}) holds. We must prove that there
exists $P$ satisfying (\ref{14-2})-(\ref{16}).

The state-space approach of the above $H^{\infty }$-control problem comes
back to solve the differential game%
\begin{equation}
\sup_{w\in L^{2}(\mathbb{R}_{+},W)}\inf_{u\in L^{2}(\mathbb{R}_{+},U)}\frac{1%
}{2}\int_{0}^{\infty }(\left\Vert z(t)\right\Vert _{Z}^{2}-\gamma
^{2}\left\Vert w(t)\right\Vert _{W}^{2})dt,  \label{11}
\end{equation}%
subject to (\ref{1})-(\ref{1'}), which ensures a prescribed bound on the
Hardy norm $H^{\infty }$ of the transfer operator (see e.g., \cite%
{VB-H-95-hyp}).

Let $J:L^{2}(\mathbb{R}_{+};U)\times L^{2}(\mathbb{R}_{+};W)\rightarrow
\lbrack -\infty ,\infty ]$ be defined as%
\begin{equation}
J(u,w)=\frac{1}{2}\int_{0}^{\infty }\{\left\Vert
C_{1}y(t)+D_{1}u(t)\right\Vert _{Z}^{2}-\gamma ^{2}\left\Vert
w(t)\right\Vert _{W}^{2}\}dt  \label{200}
\end{equation}%
and consider first a minimization problem, for a fixed $w\in L^{2}(\mathbb{R}%
_{+};W)$, 
\begin{equation}
\inf_{u\in L^{2}(\mathbb{R}_{+};U)}J(u,w),  \label{201}
\end{equation}%
subject to system (\ref{1})-(\ref{2}). By hypothesis $(i_{4})$ we see that 
\begin{equation}
J(u,w)=\frac{1}{2}\int_{0}^{\infty }\{\left\Vert C_{1}y(t)\right\Vert
_{Z}^{2}+\left\Vert u(t)\right\Vert _{U}^{2}-\gamma ^{2}\left\Vert
w(t)\right\Vert _{W}^{2}\}dt,  \label{202}
\end{equation}%
so $u\rightarrow J(u,w)$ is strictly convex, whence it easily can be shown
that (\ref{201}) has a unique solution 
\begin{equation}
u^{\ast }=\Gamma w  \label{203}
\end{equation}%
with $\Gamma :L^{2}(\mathbb{R}_{+};W)\rightarrow L^{2}(\mathbb{R}_{+};U).$

We denote by $y^{u^{\ast }}$ the solution to (\ref{1}) corresponding to $%
u^{\ast }$ (realizing the minimum in (\ref{201})) and $w,$ that is, $%
y^{u^{\ast }}:=y^{u^{\ast },w}.$

\begin{lemma}
There exists $p\in C(\mathbb{R}_{+};H)\cap L^{2}(\mathbb{R}_{+};H)$
satisfying%
\begin{equation}
p^{\prime }(t)=-A_{F}^{\ast }p(t)+C_{1}^{\ast }C_{1}y^{u^{\ast }}(t)+F^{\ast
}u^{\ast }(t),\mbox{ }t\in \mathbb{R}_{+},  \label{204}
\end{equation}%
\begin{equation}
u^{\ast }(t)=B_{2}^{\ast }p(t),\mbox{ a.e. }t>0.  \label{205}
\end{equation}
\end{lemma}

\begin{proof}
We note first that the solution to the equation $y^{\prime
}(t)=A_{F}y(t)+B_{1}w(t)$ with $y_{0}\in H,$ where $A_{F}=A+B_{2}F$ is
exponentially stable on $H$, is in $L^{2}(\mathbb{R}_{+};H).$ Indeed, 
\begin{eqnarray*}
&&\left\Vert y(t)\right\Vert _{H}\leq Ce^{-\alpha t}\left\Vert
y_{0}\right\Vert _{H}+\int_{0}^{t}\left\Vert e^{A_{F}(t-s)}w(s)\right\Vert
_{H}ds \\
&\leq &Ce^{-\alpha t}\left\Vert y_{0}\right\Vert _{H}+\int_{0}^{t}e^{-\alpha
(t-s)}\left\Vert w(s)\right\Vert _{W}ds,\mbox{ }t\geq 0,
\end{eqnarray*}%
and by applying the Young's inequality for convolution (\ref{400}) with $%
r=2, $ $p=1$ and $q=2$ we obtain%
\begin{eqnarray}
&&\int_{0}^{\infty }\left\Vert y(t)\right\Vert _{H}^{2}dt\leq C\left\{
\left\Vert y_{0}\right\Vert _{H}^{2}+\int_{0}^{\infty }\left(
\int_{0}^{t}e^{-\alpha (t-s)}\left\Vert w(s)\right\Vert _{W}ds\right)
^{2}dt\right\}  \label{11-2} \\
&\leq &C\left\Vert y_{0}\right\Vert _{H}^{2}+C\left( \int_{0}^{\infty
}e^{-\alpha t}dt\right) ^{2}\int_{0}^{\infty }\left\Vert w(t)\right\Vert
_{W}^{2}dt\leq C(\left\Vert y_{0}\right\Vert _{H}^{2}+\left\Vert
w\right\Vert _{L^{2}(0,\infty ;W)}^{2})<\infty .  \notag
\end{eqnarray}

We specify that the solution to (\ref{204}) should be understand in the
following mild sense 
\begin{equation*}
p(t)=-\int_{t}^{\infty }e^{-A_{F}^{\ast }(s-t)}(C_{1}^{\ast }C_{1}y^{u^{\ast
}}(s)+F^{\ast }u^{\ast }(s))ds,
\end{equation*}%
and so $C([0,\infty );H)$ because $F^{\ast }\in L(U,H),$ $C_{1}^{\ast
}C_{1}\in L(H,H).$ Since $A_{F}^{\ast }$ generates an analytic $C_{0}$%
-semigroup it follows by its regularizing effect that $p\in
W^{1,2}(0,T;H)\cap L^{2}(0,T;D(A_{F}^{\ast })),$ for all $T>0.$

We introduce $v:=u-Fy$ and write problem (\ref{201}) as 
\begin{equation}
\inf_{v\in L^{2}(\mathbb{R}_{+};U)}\frac{1}{2}\int_{0}^{\infty }\{\left\Vert
C_{1}y(t)\right\Vert _{Z}^{2}+\left\Vert Fy(t)+v(t)\right\Vert
_{U}^{2}-\gamma ^{2}\left\Vert w(t)\right\Vert _{W}^{2}\}dt  \label{207}
\end{equation}%
subject to $y^{\prime }(t)=A_{F}y(t)+B_{1}w(t)+B_{2}v(t),$ $t\geq 0,$ $%
y(0)=y_{0}$. Since the functional is weakly lower semicontinous and convex,
it follows that (\ref{207}) has a unique solution $v^{\ast }=u^{\ast
}-Fy^{u^{\ast }},$ with $u^{\ast }$ the solution to (\ref{201}) and $%
y^{u^{\ast }}$ the solution to (\ref{1}) corresponding to $u^{\ast }$ and $%
w. $

We set the variation $v^{\lambda }=v^{\ast }+\lambda V,$ where $\lambda >0,$ 
$V\in L^{2}(\mathbb{R}_{+};U)$ and write the system in variations 
\begin{equation}
Y^{\prime }(t)=A_{F}Y(t)+B_{2}V(t),\mbox{ }Y(0)=0,  \label{209}
\end{equation}%
where $Y(t)=\lim_{\lambda \rightarrow 0}\frac{y^{v^{\lambda }}-y^{v^{\ast }}%
}{\lambda }$ weakly in $L^{2}(\mathbb{R}_{+};H).$ Eq. (\ref{209}) can be
still written as $Y^{\prime }(t)=AY(t)+B_{2}(FY(t)+V(T))$ and so it is
easily seen that it has a unique solution $Y$ belonging to $W^{1,2}(\mathbb{R%
}_{+};(D(A_{F}^{\ast }))^{\prime })\cap L^{2}(\mathbb{R}_{+};H)$, the latter
following in the same way as shown before for $y(t)$ in (\ref{16-0})$.$

Writing that $v^{\ast }$ realizes the minimum in (\ref{207}), in particular
that $J(v^{\lambda },w)\geq J(v^{\ast },w),$ we deduce 
\begin{equation*}
\int_{0}^{\infty }\left\{ (C_{1}y^{u^{\ast }}(t),C_{1}Y(t))_{Z}+(Fy^{u^{\ast
}}(t)+v^{\ast }(t),FY(t)+V(t))_{U}\right\} dt\geq 0.
\end{equation*}%
If $\lambda \rightarrow -\lambda $ we obtain the reverse inequality, so that
in conclusion 
\begin{equation}
\int_{0}^{\infty }\left\{ (C_{1}^{\ast }C_{1}y^{u^{\ast }}(t)+F^{\ast
}Fy^{u^{\ast }}(t)+F^{\ast }v^{\ast }(t),Y(t))_{H}+(Fy^{u^{\ast
}}(t)+v^{\ast }(t),V(t))_{U}\right\} dt=0,  \label{210}
\end{equation}%
for all $V\in L^{2}(\mathbb{R}_{+};U).$ By testing the first equation (\ref%
{209}) by $p(t)\in D(A_{F}^{\ast }),$ solution to (\ref{204}) and
integrating with respect to $t$ from $0$ to $\infty ,$ we obtain 
\begin{equation}
\int_{0}^{\infty }(p^{\prime }(t)+A_{F}^{\ast
}p(t),Y(t))_{H}dt+\int_{0}^{\infty }(B_{2}^{\ast }p(t),V(t))_{U}dt=0,
\label{211}
\end{equation}%
which by (\ref{204}) yields%
\begin{equation}
\int_{0}^{\infty }\left( C_{1}^{\ast }C_{1}y^{u^{\ast }}(t)+F^{\ast }u^{\ast
}(t),Y(t)\right) _{H}dt=-\int_{0}^{\infty }(B_{2}^{\ast }p(t),V(t))_{U}dt.
\label{212}
\end{equation}%
By comparison with (\ref{210}), where we write $v^{\ast }=u^{\ast
}-Fy^{u^{\ast }},$ this yields 
\begin{equation}
\int_{0}^{\infty }(-B_{2}^{\ast }p(t)+u^{\ast }(t),V(t))_{U}dt=0,\mbox{ for
all }V\in L^{2}(\mathbb{R}_{+};U).  \label{213}
\end{equation}%
Therefore, we obtain (\ref{205}) as claimed.
\end{proof}

\noindent Then, the dual system (\ref{204}) can be still written by the
replacement of $v^{\ast }$ as 
\begin{equation}
p^{\prime }(t)=-A^{\ast }p(t)+C_{1}^{\ast }C_{1}y^{u^{\ast }}(t),\mbox{ a.e. 
}t\in \mathbb{R}_{+}.  \label{214-2}
\end{equation}

Now, let us consider the function $\varphi :L^{2}(\mathbb{R}%
_{+};W)\rightarrow \mathbb{R}_{+},$ $\varphi (w)=-J(\Gamma w,w),$that is 
\begin{equation*}
\varphi (w)=\frac{1}{2}\int_{0}^{\infty }\left( \gamma ^{2}\left\Vert
w(t)\right\Vert _{W}^{2}-\left\Vert C_{1}y^{u^{\ast }}(t)+D_{1}u^{\ast
}(t)\right\Vert _{Z}^{2}\right) dt,
\end{equation*}%
where $y^{u^{\ast }}$ is the solution to (\ref{1}) corresponding to $%
(u^{\ast },w).$ By (\ref{8-1}) and (\ref{8-2}) we have 
\begin{equation*}
C_{1}y^{u^{\ast }}(t)+D_{1}u^{\ast }(t)=G_{F}w(t)-f_{0}(t)
\end{equation*}%
and so 
\begin{eqnarray*}
&&\left\Vert (C_{1}y^{u^{\ast }}(t)+D_{1}u^{\ast }(t)\right\Vert
_{Z}^{2}=\left\Vert G_{F}w(t)\right\Vert
_{Z}^{2}-2(G_{F}w(t),f_{0}(t))_{Z}+\left\Vert f_{0}(t)\right\Vert _{Z}^{2} \\
&\leq &(1+\delta )\left\Vert G_{F}w(t)\right\Vert _{Z}^{2}+C_{\delta
}\left\Vert f_{0}(t)\right\Vert _{Z}^{2},\mbox{ }\forall t\geq 0.
\end{eqnarray*}%
Now, we integrate from $0$ to $\infty ,$ note that $f_{0}\in L^{2}(\mathbb{R}%
_{+};H),$ and get, 
\begin{equation*}
\int_{0}^{\infty }\left\Vert (C_{1}y^{u^{\ast }}(t)+D_{1}u^{\ast
}(t)\right\Vert _{Z}^{2}dt\leq (1+\delta )(\gamma ^{2}-\varepsilon
)\int_{0}^{\infty }\left\Vert w(t)\right\Vert ^{2}+C_{\delta },
\end{equation*}%
where $\varepsilon $ is fixed and the last inequality is implied by (\ref%
{14-1}). We can find $\delta $ and $\widetilde{\delta }$ such that $%
(1+\delta )(\gamma ^{2}-\varepsilon )\leq \gamma ^{2}-\widetilde{\delta },$
which is verified with the choice $\widetilde{\delta }<\varepsilon -\delta
(\gamma ^{2}-\varepsilon )$ and $\delta <\frac{\varepsilon }{\gamma
^{2}-\varepsilon }.$ Then\thinspace\ 
\begin{equation*}
\varphi (w)\geq \widetilde{\delta }\int_{0}^{\infty }\left\Vert
w(t)\right\Vert _{W}^{2}dt+C,\mbox{ }
\end{equation*}%
and it turns out that $\varphi $ attains its minimum on $L^{2}(\mathbb{R}%
_{+};W)$ in a unique point $w^{\ast }.$

\begin{lemma}
We have 
\begin{equation}
w^{\ast }(t)=-\gamma ^{-2}B_{1}^{\ast }p(t),\mbox{ a.e. }t>0,  \label{215}
\end{equation}%
where $p\in W^{1,2}(0,T;H)$ is the solution to (\ref{214-2}).
\end{lemma}

\begin{proof}
Recall that $u^{\ast }=\Gamma w$ and that $y^{u^{\ast }}$ satisfies the
problem 
\begin{equation*}
(y^{u^{\ast }})^{\prime }(t)=Ay^{u^{\ast }}(t)+B_{1}w(t)+B_{2}\Gamma w(t),%
\mbox{ }t\in \mathbb{R}_{+},\mbox{ }y^{\ast }(0)=y_{0}
\end{equation*}%
and proceed by giving variations to $w,$ that is $w^{\lambda }=w^{\ast
}+\lambda \widetilde{w},$ $w\in L^{2}(\mathbb{R}_{+};H).$ Then, the system
in variations is 
\begin{equation}
Y^{\prime }(t)=AY(t)+B_{1}\widetilde{w}(t)+B_{2}\Gamma \widetilde{w}(t)\mbox{%
, }t\in \mathbb{R}_{+},\mbox{ }Y(0)=0  \label{215-1}
\end{equation}%
and the condition of optimality reads 
\begin{equation}
\int_{0}^{\infty }(\gamma ^{2}w^{\ast }(t)-\Gamma ^{\ast }\Gamma w^{\ast
}(t),\widetilde{w}(t))_{W}dt-\int_{0}^{\infty }(C_{1}^{\ast }C_{1}y^{u^{\ast
}}(t),Y(t))_{H}dt=0,  \label{215-3}
\end{equation}%
for all $\widetilde{w}\in L^{2}(\mathbb{R}_{+};W).$ Let us recall the dual
system (\ref{214-2}) and test (\ref{215-1}) by $p(t)$ and integrate for $%
t\in (0,\infty ).$ We get 
\begin{equation}
\int_{0}^{\infty }(p^{\prime }(t)+A^{\ast }p(t),Y(t))_{H}dt+\int_{0}^{\infty
}(B_{1}^{\ast }p(t)+\Gamma ^{\ast }B_{2}^{\ast }p(t),\widetilde{w}%
(t))_{W}dt=0.  \label{215-4}
\end{equation}%
The latter and (\ref{215-3}) gives 
\begin{equation*}
\int_{0}^{\infty }(\gamma ^{2}w^{\ast }(t)-\Gamma ^{\ast }\Gamma w^{\ast
}(t),\widetilde{w}(t))_{W}dt+\int_{0}^{\infty }(B_{1}^{\ast }p(t)+\Gamma
^{\ast }B_{2}^{\ast }p(t),\widetilde{w}(t))_{W}dt=0,
\end{equation*}%
so that, since $-\Gamma ^{\ast }\Gamma w^{\ast }(t)+\Gamma ^{\ast
}B_{2}^{\ast }p(t)=-\Gamma ^{\ast }u^{\ast }(t)+\Gamma ^{\ast }u^{\ast
}(t)=0,$ we obtain 
\begin{equation*}
\int_{0}^{\infty }(\gamma ^{2}w^{\ast }(t)+B_{1}^{\ast }p(t),\widetilde{w}%
(t))_{W}dt=0,
\end{equation*}%
for all $\widetilde{w}\in L^{2}(\mathbb{R}_{+};W),$ that implies (\ref{215}%
), as claimed.
\end{proof}

Thus, we have proved that (\ref{11}) has a unique solution $(u^{\ast
},w^{\ast })$ with the corresponding state denoted $y^{\ast },$
characterized by the Euler-Lagrange system 
\begin{equation}
y^{\ast \prime }(t)=Ay^{\ast }(t)+B_{1}w^{\ast }(t)+B_{2}u^{\ast }(t),\mbox{ 
}t\in \mathbb{R}_{+},\mbox{ }y^{\ast }(0)=y_{0},  \label{216}
\end{equation}%
\begin{equation}
p^{\prime }(t)=-A^{\ast }p(t)+C_{1}^{\ast }C_{1}y^{\ast }(t),\mbox{ }t\in 
\mathbb{R}_{+},  \label{217}
\end{equation}%
\begin{equation}
u^{\ast }(t)=B_{2}^{\ast }p(t),\mbox{ a.e. }t>0.  \label{218}
\end{equation}%
\begin{equation}
w^{\ast }(t)=-\gamma ^{-2}B_{1}^{\ast }p(t),\mbox{ a.e. }t>0,  \label{219}
\end{equation}%
where we already know that 
\begin{eqnarray*}
y^{\ast } &\in &C([0,\infty );(D(A^{\ast }))^{\prime })\cap L^{2}(0,T;H),%
\mbox{ }\forall T>0, \\
p &\in &C([0,\infty );H)\cap L^{2}(\mathbb{R}_{+};H).
\end{eqnarray*}

\begin{lemma}
Let $y_{0}\in H.$ Then, 
\begin{equation}
y^{\ast }\in C([0,\infty );H)\cap W^{1,2}(\delta ,T;H),\mbox{ }\forall
\delta ,\mbox{ }0<\delta \leq T<\infty ,  \label{3-35-2}
\end{equation}%
\begin{equation}
p\in W^{1,2}(0,T;H)\cap L^{2}(0,T;D(A^{\ast })),\mbox{ }\forall T>0.\mbox{ }
\label{3-35-1}
\end{equation}
\end{lemma}

\begin{proof}
Since $A^{\ast }$ generates an analytic $C_{0}$-semigroup and $C_{1}^{\ast
}C_{1}y^{\ast }\in L^{2}(\mathbb{R}_{+};H)$ we see by (\ref{217}) that (\ref%
{3-35-1}) holds. Moreover, by (\ref{216}) and (\ref{218}) we have%
\begin{eqnarray}
y^{\ast }(t) &=&e^{At}y_{0}+\int_{0}^{t}e^{A(t-s)}B_{1}w^{\ast
}(s)ds+\int_{0}^{t}e^{A(t-s)}B_{2}B_{2}^{\ast }p(s)ds  \label{3-35-5} \\
&=&e^{At}y_{0}+g_{1}(t)+g_{2}(t),\mbox{ }\forall t\geq 0.  \notag
\end{eqnarray}%
The first two terms are in $C([0,\infty );H)\cap W^{1,2}(\delta ,T;H).$ By (%
\ref{3-35-1}), $B_{2}B_{2}^{\ast }p\in W^{1,2}(0,T;(D(A^{\ast }))^{\prime })$
and so we may represent it as $B_{2}B_{2}^{\ast }p=(A-\omega I)f,$ with $%
f\in W^{1,2}(0,T;H),$ for $\omega $ sufficiently large. This yields%
\begin{eqnarray*}
g_{2}(t) &=&\int_{0}^{t}e^{A(t-s)}(A-\omega I)f(s)ds=-\omega
\int_{0}^{t}e^{A(t-s)}f(s)ds-\int_{0}^{t}\left( \frac{d}{ds}%
e^{A(t-s)}\right) f(s)ds \\
&=&-\omega
\int_{0}^{t}e^{A(t-s)}f(s)ds-f(t)+e^{At}f(0)+\int_{0}^{t}e^{A(t-s)}f^{\prime
}(s)ds,\mbox{ }\forall t\geq 0.
\end{eqnarray*}%
Since $e^{At}$ is an analytic semigroup it follows that $g(t)=%
\int_{0}^{t}e^{A(t-s)}f^{\prime }(s)ds$, the solution to $g^{\prime
}(t)=Ag(t)+f(t),$ $g(0)=0\in D(A),$ is in $W^{1,2}(0,T;H),$ as the first two
terms. Though $f(0)\notin D(A)$, the third term is in $C([0,\infty );H)\cap
W^{1,2}(\delta ,T;H),$ $\forall 0<\delta \leq T<\infty $ and so is $g_{2}$
and $y^{\ast },$ too. Moreover, since $A^{\ast }$ is analytic, then (\ref%
{3-35-1}) holds.
\end{proof}

\noindent \textit{Proof} (of Theorem \ref{Th-main}, continued). Now we set 
\begin{equation}
Py_{0}:=-p(0),\mbox{ for }y_{0}\in H  \label{220}
\end{equation}%
and note that $P\in L(H,H).$

Moreover, by adding (\ref{216}) multiplied by $p(t)$ with (\ref{217})
multiplied by $y^{\ast }(t)$ and integrating on $(0,\infty )$ we get%
\begin{eqnarray}
-2(y_{0},p(0))_{H} &=&\int_{0}^{\infty }\left\{ \left\langle Ay^{\ast
}(t),p(t)\right\rangle _{(D(A^{\ast }))^{\prime },D(A^{\ast })}+(w^{\ast
}(t),B_{1}^{\ast }p(t))_{W}+(u^{\ast }(t),B_{2}^{\ast }p(t))_{U}\right\} dt 
\notag \\
&&+\int_{0}^{\infty }\left\{ -\left\langle Ay^{\ast }(t),p(t)\right\rangle
_{(D(A^{\ast }))^{\prime },D(A^{\ast })}+(C_{1}^{\ast }C_{1}y^{\ast
}(t),p(t))_{H}\right\} dt  \label{220-1} \\
&=&\int_{0}^{\infty }\left\{ -\gamma ^{2}\left\Vert w^{\ast }(t)\right\Vert
_{W}^{2}+\left\Vert u^{\ast }(t)\right\Vert _{U}^{2}+\left\Vert C_{1}y^{\ast
}(t)\right\Vert _{Z}^{2}\right\} dt  \notag
\end{eqnarray}%
whence 
\begin{eqnarray*}
(Py_{0},y_{0})_{H} &=&-(p(0),y_{0})_{H}=\frac{1}{2}\int_{0}^{\infty }\left(
\left\Vert C_{1}y^{\ast }(t)\right\Vert _{Z}^{2}+\left\Vert u^{\ast
}(t)\right\Vert _{U}^{2}-\gamma ^{2}\left\Vert w^{\ast }(t)\right\Vert
_{W}^{2}\right) dt \\
&=&\sup_{w\in L^{2}(\mathbb{R}_{+};W)}\inf_{u\in L^{2}(\mathbb{R}_{+};U)}%
\frac{1}{2}\int_{0}^{\infty }\left( \left\Vert C_{1}y(t)\right\Vert
_{Z}^{2}+\left\Vert u(t)\right\Vert _{U}^{2}-\gamma ^{2}\left\Vert
w(t)\right\Vert _{W}^{2}\right) dt \\
&\geq &\inf_{u\in L^{2}(\mathbb{R}_{+};U)}\frac{1}{2}\int_{0}^{\infty
}\left( \left\Vert C_{1}y(t)\right\Vert _{Z}^{2}+\left\Vert u(t)\right\Vert
_{U}^{2}\right) dt\geq 0,
\end{eqnarray*}%
hence $P\geq 0.$

Moreover, $P=P^{\ast }.$ Indeed, let $y_{0},$ $z_{0}\in H$ and $(y^{\ast
},p),$ $(z^{\ast },q)$ be the corresponding solutions to (\ref{216})-(\ref%
{219}). Namely, $(z^{\ast },q)$ satisfy 
\begin{equation*}
z^{\ast \prime }(t)=Az^{\ast }(t)+B_{1}w^{\ast }(t)+B_{2}u^{\ast }(t),\mbox{ 
}t\in \mathbb{R}_{+},\mbox{ }y^{\ast }(0)=y_{0},
\end{equation*}%
\begin{equation*}
q^{\prime }(t)=-A^{\ast }q(t)+C_{1}^{\ast }C_{1}z^{\ast }(t),\mbox{ }t\in 
\mathbb{R}_{+}.
\end{equation*}%
We see that 
\begin{equation*}
\frac{d}{dt}(p(t),z^{\ast }(t))_{H}=\frac{d}{dt}(q(t),y^{\ast }(t))_{H},%
\mbox{ }\forall t\geq 0
\end{equation*}%
and this yields $(Py_{0},z_{0})_{H}=(y_{0},Pz_{0})_{H},$ as claimed.

We recall that by the dynamic programming principle (see e.g., \cite%
{VB-book-94}, p. 104), the minimization problem (\ref{201}) for $w=w^{\ast
}, $ is equivalent with the following problem 
\begin{equation*}
\inf_{u\in L^{2}(\mathbb{R}_{+};U)}\frac{1}{2}\int_{t}^{\infty }\left(
\left\Vert C_{1}y(s)\right\Vert _{Z}^{2}+\left\Vert u(s)\right\Vert
_{U}^{2}-\gamma ^{2}\left\Vert w^{\ast }(s)\right\Vert _{W}^{2}\right) ds
\end{equation*}%
subject to (\ref{1})-(\ref{2}) in $S_{t}=\{(t,\infty );$ $y(t)=y^{\ast
}(t)\},$ for every $t\geq 0.$ Since $u^{\ast }$ is the solution to this
problem it follows by (\ref{220}) that 
\begin{equation}
p(t)=-Py^{\ast }(t),\mbox{ }\forall t\geq 0.  \label{221}
\end{equation}%
We denote by $T_{P}(t):H\rightarrow H$ the family of operators 
\begin{equation}
T_{P}(t)y_{0}=y^{\ast }(t),\mbox{ }\forall t\geq 0  \label{222}
\end{equation}%
where $y^{\ast }(t)$ is the solution to (\ref{216}) with $u^{\ast }$ and $%
w^{\ast }$ given by (\ref{217})-(\ref{219}). By (\ref{3-35-2}) it follows
that $T_{P}(t)$ is a $C_{0}$-semigroup on $H.$

Let us denote by $A_{P}$ the infinitesimal generator of $T_{P}(t),$ that is%
\begin{equation}
\frac{dy^{\ast }}{dt}(t)=A_{P}y^{\ast }(t),\mbox{ }\forall t\geq 0,\mbox{ }%
y^{\ast }(0)=y_{0},  \label{222-1}
\end{equation}%
or, equivalently 
\begin{equation}
y^{\ast }(t)=e^{A_{P}t}y_{0},\mbox{ }t\geq 0,\mbox{ }\forall y_{0}\in H.
\label{222-2}
\end{equation}%
If $y_{0}\in D(A_{P})$ we have 
\begin{equation}
y^{\ast }\in C^{1}([0,T];H)\cap C([0,T];D(A_{P})),\mbox{ }\forall T>0.
\label{3-43-1}
\end{equation}%
Here, $D(A_{P})=\{y\in H;$ $A_{P}y\in H\}$ is the domain of $A_{P}$. The
space $\mathcal{X}$ in Theorem \ref{Th-main} is actually%
\begin{equation}
\mathcal{X}:=D(A_{P}).  \label{3-43-2}
\end{equation}%
Now, replacing in the right-hand side of (\ref{216}) $u^{\ast }$ and $%
w^{\ast }$ by (\ref{218})-(\ref{219}), (\ref{217}) and (\ref{221}) we get 
\begin{equation}
y^{\ast \prime }(t)=\widetilde{\Lambda _{P}}y^{\ast }(t)  \label{216-1}
\end{equation}%
where $\widetilde{\Lambda _{P}}$ is the operator 
\begin{equation*}
\widetilde{\Lambda _{P}}:H\rightarrow (D(A^{\ast }))^{\prime },\mbox{ }%
\widetilde{\Lambda _{P}}y=Ay-B_{2}B_{2}^{\ast }Py+\gamma
^{-2}B_{1}B_{1}^{\ast }Py\in (D(A^{\ast }))^{\prime }
\end{equation*}%
and $A$ is the extension from $H$ to $(D(A^{\ast }))^{\prime }.$

We define by $\Lambda _{P}:D(\Lambda _{P})\subset H\rightarrow H$ the
restriction of the operator $\widetilde{\Lambda _{P}}$ to $H,$ namely 
\begin{eqnarray}
\Lambda _{P}y &=&(A-B_{2}B_{2}^{\ast }P+\gamma ^{-2}B_{1}B_{1}^{\ast }P)y,%
\mbox{ }y\in D(\Lambda _{P}),  \label{222-3} \\
D(\Lambda _{P}) &=&\{y\in H;\mbox{ }(A-B_{2}B_{2}^{\ast }P+\gamma
^{-2}B_{1}B_{1}^{\ast }P)y\in H\}.  \notag
\end{eqnarray}

\begin{lemma}
\label{operators}We have 
\begin{equation}
P\in L(\mathcal{X},D(A^{\ast })),  \label{223}
\end{equation}%
\begin{equation}
B_{2}^{\ast }P\in L(\mathcal{X};U),  \label{223-1}
\end{equation}%
\begin{equation}
A_{P}y=\Lambda _{P}y,\mbox{ for all }y\in \mathcal{X}\subset D(\Lambda _{P})
\label{224}
\end{equation}%
and $\Lambda _{P}$ generates a $C_{0}$-semigroup on $H.$

Moreover, if $\Lambda _{P}$ is closed in $H$, then 
\begin{equation}
\mathcal{X}=D(A_{P})=D(\Lambda _{P}).  \label{224-1}
\end{equation}
\end{lemma}

\begin{proof}
Let $y_{0}\in D(A_{P})$. We know by (\ref{3-43-1}) that $y^{\ast }\in
C^{1}([0,T];H)$ and $A_{P}y^{\ast }\in C([0,T];H)$ for all $T>0,$ and so by (%
\ref{217}) it follows therefore that $p^{\prime }\in C([0,T];H)$ and so $%
A^{\ast }p\in C([0,T];H).$ Hence, $A^{\ast }p(0)\in H.$ It follows that $%
p(0)\in D(A^{\ast })$ and so $Py_{0}\in D(A^{\ast }).$ This implies (\ref%
{223}). Since $P\in (\mathcal{X},D(A^{\ast }))$ and $B_{2}^{\ast }\in
L(D(A^{\ast }),U)$ it follows (\ref{223-1}).

We have by (\ref{222-1}) and (\ref{3-43-1}) that 
\begin{equation*}
\frac{d}{dt}(y^{\ast }(t),\varphi )_{H}=(A_{P}y^{\ast }(t),\varphi )_{H},%
\mbox{ }\forall t\geq 0,\mbox{ }\varphi \in H.
\end{equation*}%
On the other hand, by (\ref{216-1}) we have (see the weak form (\ref{13-1})
applied to $\widetilde{\Lambda _{P}}:H\rightarrow (D(A^{\ast }))^{\prime }$) 
\begin{equation*}
\frac{d}{dt}(y^{\ast }(t),\varphi )_{H}=\left\langle \frac{dy^{\ast }}{dt}%
(t),\varphi )\right\rangle _{(D(A^{\ast }))^{\prime },D(A^{\ast
})}=\left\langle \widetilde{\Lambda _{P}}y^{\ast }(t),\varphi \right\rangle
_{(D(A^{\ast }))^{\prime },D(A^{\ast })},\mbox{ }\forall t\geq 0,\mbox{ }%
\forall \varphi \in D(A^{\ast }).
\end{equation*}%
Hence, 
\begin{equation*}
(A_{P}y^{\ast }(t),\varphi )_{H}=\left\langle \widetilde{\Lambda _{P}}%
y^{\ast }(t),\varphi \right\rangle _{(D(A^{\ast }))^{\prime },D(A^{\ast })},%
\mbox{ }\forall t\geq 0,\mbox{ }\forall \varphi \in D(A^{\ast }).
\end{equation*}%
Recalling that $y^{\ast }\in C^{1}([0,\infty );H)\subset C([0,\infty
);(D(A^{\ast }))^{\prime })$ and letting $t\rightarrow 0$ we get 
\begin{equation*}
(A_{P}y_{0},\varphi )_{H}=\left\langle \widetilde{\Lambda _{P}}y_{0},\varphi
\right\rangle _{(D(A^{\ast }))^{\prime },D(A^{\ast })},\mbox{ }\forall
\varphi \in D(A^{\ast }).
\end{equation*}%
This implies that $\widetilde{\Lambda _{P}}y_{0}\in H,$ namely $y_{0}\in
D(\Lambda _{P}),$ and $A_{P}y_{0}=\Lambda _{P}y_{0}$ on $D(A_{P})\subset
D(\Lambda _{P}),$ that is (\ref{224}).

Since these two operators coincide on $D(A_{P})$ then $\Lambda _{P}$
generates a $C_{0}$-semigroup on $H.$

Now, $D(A_{P})\subset D(\Lambda _{P})\subset H$ and since $D(A_{P})$ is
dense in $H$ it follows that $D(\Lambda _{P})$ is dense in $H$ and $D(A_{P})$
is dense in $D(\Lambda _{P}).$

Assume that $\Lambda _{P}$ is closed and let $y_{0}\in D(\Lambda _{P}).$
There exists $(y_{0}^{n})_{n}\subset D(A_{P})$, $y_{0}^{n}\rightarrow y_{0}$
in $H$ and by (\ref{224}) we have%
\begin{equation*}
(A_{P}y_{0}^{n},\varphi )_{H}=\left( \Lambda _{P}y_{0}^{n},\varphi \right)
_{H},\mbox{ }\varphi \in H,
\end{equation*}%
which implies (using the adjoint of $A_{P}^{\ast }$ which is the generator
of a $C_{0}$-semigroup on $H)$ that 
\begin{equation*}
\left( y_{0}^{n},A_{P}^{\ast }\varphi \right) _{H}=\left( \Lambda
_{P}y_{0}^{n},\varphi \right) _{H},\mbox{ }\varphi \in D(A_{P}^{\ast
})\subset H.
\end{equation*}%
Since $\Lambda _{P}$ is closed, by letting $n\rightarrow \infty $ we obtain%
\begin{equation*}
(y_{0},A_{P}^{\ast }\varphi )_{H}=\left( \Lambda _{P}y_{0},\varphi \right)
_{H},\mbox{ }\varphi \in D(A_{P}^{\ast }).
\end{equation*}%
Then, $\varphi \rightarrow (y_{0},A_{P}^{\ast }\varphi )_{H}$ is a linear
continuous functional on $H$ and $\left\vert (y_{0},A_{P}^{\ast }\varphi
)_{H}\right\vert \leq C\left\Vert \varphi \right\Vert _{H},$ so that $%
y_{0}\in D(A_{P})$ and (\ref{224}) is proved.
\end{proof}

\begin{proof}
(of Theorem \ref{Th-main}, continued). To prove that $P$ is a solution to
the Riccati equation (\ref{15}) we use the relation 
\begin{equation*}
\frac{d}{dt}(y^{\ast }(t),p(t))_{H}=\left\langle (y^{\ast })^{\prime
}(t),p(t)\right\rangle _{(D(A^{\ast }))^{\prime },D(A^{\ast })}+(y^{\ast
}(t),p^{\prime }(t))_{H}
\end{equation*}%
and calculate by (\ref{216})-(\ref{219}) and (\ref{221}) a relation as done
for (\ref{220-1}) but integrating from $t$ to $\infty $. We get 
\begin{eqnarray*}
(Py^{\ast }(t),y^{\ast }(t))_{H} &=&(-p(t),y^{\ast }(t))_{H} \\
&=&\frac{1}{2}\int_{t}^{\infty }\left( \left\Vert C_{1}y^{\ast
}(t)\right\Vert _{Z}^{2}+\left\Vert u^{\ast }(t)\right\Vert _{U}^{2}-\gamma
^{2}\left\Vert w^{\ast }(t)\right\Vert _{W}^{2}\right) dt,\mbox{ }t\geq 0.
\end{eqnarray*}%
If $y_{0}\in D(A_{P})$ this implies by differentiating (by using (\ref{222-1}%
) and (\ref{3-43-1})) that 
\begin{eqnarray*}
&&(Py^{\ast }(t),A_{P}y^{\ast }(t))_{H}+(PA_{P}y^{\ast }(t),y^{\ast
}(t))_{H}+\left\Vert C_{1}y^{\ast }(t)\right\Vert _{Z}^{2} \\
&&+\left\Vert B_{2}^{\ast }Py^{\ast }(t)\right\Vert _{U}^{2}-\gamma
^{2}\left\Vert \gamma ^{-2}B_{1}^{\ast }Py^{\ast }(t)\right\Vert _{W}^{2}=0,%
\mbox{ }t\geq 0
\end{eqnarray*}%
and since, by (\ref{223}), $B_{2}^{\ast }P\in L(D(A_{P}),D(A^{\ast }))$ we
obtain for $t\rightarrow 0$ the equation 
\begin{equation}
2(Py_{0},A_{P}y_{0})_{H}+\left\Vert C_{1}y_{0}\right\Vert
_{Z}^{2}+\left\Vert B_{2}^{\ast }Py_{0}\right\Vert _{U}^{2}-\gamma
^{-2}\left\Vert B_{1}^{\ast }Py_{0}\right\Vert _{W}^{2}=0,\mbox{ }\forall
y_{0}\in D(A_{P}).  \label{225}
\end{equation}%
By differentiating along $z\in D(A_{P})$ we get 
\begin{equation*}
(Py_{0},A_{P}z)_{H}+(Pz,A_{P}y_{0})_{H}+((B_{2}B_{2}^{\ast }-\gamma
^{-2}B_{1}B_{1}^{\ast })Pz,Py_{0})_{H}+(C_{1}^{\ast }C_{1}y_{0},z)_{H}=0,
\end{equation*}%
for all $y_{0,}$ $z\in D(A_{P}).$ But here $A_{P}y=\Lambda _{P}y$ for $y\in
D(A_{P})$ and we can replace $A_{P}$ by $\Lambda _{P}$ in the previous
equation obtaining after all calculations 
\begin{equation*}
\left( A^{\ast }Py_{0},z\right) _{H}+\left( P(A-B_{2}B_{2}^{\ast }+\gamma
^{-2}B_{1}B_{1}^{\ast })Py_{0},z\right) _{H}+(C_{1}^{\ast
}C_{1}y_{0},z)_{H}=0
\end{equation*}%
for all $y_{0,}$ $z\in D(A_{P}),$ namely (\ref{15}).

For proving that the semigroup $e^{\Lambda _{P}t\mbox{ }}$is exponentially
stable we use the detectability assumption $(i_{3}).$ Let us take $K\in
L(Z,H)$ and write eq. (\ref{216}) in the following form 
\begin{equation*}
y^{\ast \prime }(t)=(A+KC_{1})y^{\ast }(t)+B_{2}u^{\ast }(t)+B_{1}w^{\ast
}(t)-KC_{1}y^{\ast }(t),\mbox{ }t\geq 0,
\end{equation*}%
or equivalently, 
\begin{eqnarray*}
y^{\ast }(t)
&=&e^{(A+KC_{1})t}y_{0}+\int_{0}^{t}e^{(A+KC_{1})(t-s)}(B_{2}u^{\ast
}(s)+B_{1}w^{\ast }(s))ds \\
&&-\int_{0}^{t}e^{(A+KC_{1})(t-s)}KC_{1}y^{\ast }(s)ds,\mbox{ for all }t\geq
0.
\end{eqnarray*}%
Since $B_{1}w^{\ast },$ $KC_{1}y^{\ast }\in L^{2}(\mathbb{R}_{+};H)$ and $%
e^{(A+KC_{1})t}$ is exponentially stable it remains to show that 
\begin{equation}
t\rightarrow \int_{0}^{t}e^{(A+KC_{1})(t-s)}B_{2}u^{\ast }(s)ds\in L^{2}(%
\mathbb{R}_{+};H).  \label{225-1}
\end{equation}%
To this end, for each $\psi \in L^{2}(\mathbb{R}_{+};H),$ using the Young's
inequality (\ref{400}) (with $p=1,$ $q=r=2)$ and (\ref{12-0}) we calculate 
\begin{eqnarray*}
&&\int_{0}^{\infty }\left( \psi
(t),\int_{0}^{t}e^{(A+KC_{1})(t-s)}B_{2}u^{\ast }(s)ds\right) _{H}dt \\
&=&\int_{0}^{\infty }\left( \int_{s}^{\infty }B_{2}^{\ast }e^{(A^{\ast
}+C_{1}^{\ast }K^{\ast })(t-s)}\psi (t)dt,u^{\ast }(s)\right) _{U}ds \\
&\leq &\left( \int_{0}^{\infty }\left( \int_{s}^{\infty }\left\Vert
B_{2}^{\ast }e^{(A^{\ast }+C_{1}^{\ast }K^{\ast })(t-s)}\psi (t)\right\Vert
_{U}dt\right) ^{2}ds\right) ^{1/2}\left( \int_{0}^{\infty }\left\Vert
u^{\ast }(s)\right\Vert _{U}^{2}ds\right) ^{1/2} \\
&\leq &\left\Vert u^{\ast }\right\Vert _{L^{2}(0,\infty ;U)}\left(
\int_{0}^{\infty }\left( \int_{0}^{\infty }\left\Vert B_{2}^{\ast
}e^{(A^{\ast }+C_{1}^{\ast }K^{\ast })(t-s)}\right\Vert _{L(H,U)}\left\Vert
\psi (t)\right\Vert _{H}dt\right) ^{2}ds\right) ^{1/2} \\
&\leq &\left\Vert u^{\ast }\right\Vert _{L^{2}(0,\infty ;U)}\left(
\int_{0}^{\infty }\left\Vert B_{2}^{\ast }e^{(A^{\ast }+C_{1}^{\ast }K^{\ast
})s}\right\Vert _{L(H,U)}ds\right) \left( \int_{0}^{\infty }\left\Vert \psi
(s)\right\Vert _{H}^{2}ds\right) ^{1/2} \\
&\leq &C\left\Vert u^{\ast }\right\Vert _{L^{2}(0,\infty ;U)}\left\Vert \psi
\right\Vert _{L^{2}(0,\infty ;U)}\leq C_{1}\left\Vert \psi \right\Vert
_{L^{2}(0,\infty ;U)},
\end{eqnarray*}%
and this implies (\ref{225-1}), as claimed.

We shall prove now that the operator $\Lambda _{P}^{1}:=A-B_{2}B_{2}^{\ast }P
$ generates an exponentially stable $C_{0}$-semigroup in $H$ with the domain 
$\{y\in H;$ $(A-B_{2}B_{2}^{\ast }P)y\in H\}=D(A_{P}).$ The solution $%
y^{\ast }(t)$ to (\ref{216}) is in $L^{2}(\mathbb{R}_{+};H)$ can be written
also as%
\begin{equation*}
y^{\ast }(t)=e^{\Lambda _{P}^{1}t}y_{0}+\gamma ^{-2}\int_{0}^{t}e^{\Lambda
_{P}^{1}(t-s)}B_{1}B_{1}^{\ast }Py^{\ast }(s)ds
\end{equation*}%
and since the second term on the right-hand side is in $L^{2}(\mathbb{R}%
_{+};H),$ it follows that $e^{\Lambda _{P}^{1}t}y_{0}\in L^{2}(\mathbb{R}%
_{+};H).$

Now, we shall prove (\ref{14-1}). Let us consider the equation 
\begin{equation}
y^{\prime }(t)=(A-B_{2}B_{2}^{\ast }P)y(t)+B_{1}w(t),\mbox{ }t\geq 0,\mbox{ }%
y(0)=0,  \label{225-2}
\end{equation}%
with $w\in L^{2}(\mathbb{R}_{+};W).$ As seen earlier, this equation has a
unique mild solution and by (\ref{13-1}) we have%
\begin{equation}
\frac{d}{dt}(y(t),\varphi )_{H}=(y(t),(A^{\ast }-B_{2}B_{2}^{\ast }P)\varphi
)_{H}+(B_{1}w(y),\varphi )_{H},\mbox{ }\forall \varphi \in D(A^{\ast }).
\label{3-49-1}
\end{equation}%
Let $p(t)=-Py(t),$ $t>0.$ Since by (\ref{223}) $P\in L(D(A_{P}),D(A^{\ast
})) $ it follows that $p(t)\in D(A^{\ast })$ and $p$ is the solution to eq. (%
\ref{217}) with $y^{\ast }$ replaced by $y.$ Moreover, as seen earlier by (%
\ref{217}) it follows that $A^{\ast }p,$ $p^{\prime }\in L^{2}(\mathbb{R}%
_{+};H)$ and we have by (\ref{3-49-1})%
\begin{equation*}
\frac{d}{dt}(y(t),p(t))_{H}=(y(t),(A^{\ast }-B_{2}B_{2}^{\ast
}P)p(t))_{H}+(B_{1}w(t),p(t))_{H}+(y^{\prime }(t),p(t))_{H}.
\end{equation*}%
Then we calculate using (\ref{221}) and (\ref{15}) 
\begin{eqnarray*}
&&\frac{d}{dt}(Py(t),y(t))_{H}=2(Py(t),y^{\prime }(t))_{H} \\
&=&2(Py(t),Ay(t))_{H}-2\left\Vert B_{2}^{\ast }Py(t)\right\Vert
_{H}^{2}+2(B_{1}w(t),Py(t))_{H} \\
&=&\left\Vert B_{2}^{\ast }Py(t)\right\Vert _{H}^{2}-\gamma ^{-2}\left\Vert
B_{1}^{\ast }Py(t)\right\Vert _{H}^{2}-\left\Vert C_{1}y(t)\right\Vert
_{H}^{2}-2\left\Vert B_{2}^{\ast }Py(t)\right\Vert
_{H}^{2}+2(B_{1}w(t),Py(t))_{H} \\
&=&-\left\Vert B_{2}^{\ast }Py(t)\right\Vert _{H}^{2}-\left\Vert
C_{1}y(t)\right\Vert _{H}^{2}-\gamma ^{-2}\left\Vert B_{1}^{\ast
}Py(t)\right\Vert _{H}^{2}+2(w(t),B_{1}^{\ast }Py(t))_{W},\mbox{ a.e. }t>0.
\end{eqnarray*}%
Integrating this from $0$ to $\infty $ we obtain 
\begin{equation*}
0=\int_{0}^{\infty }\left( -\left\Vert B_{2}^{\ast }Py(t)\right\Vert
_{H}^{2}-\left\Vert C_{1}y(t)\right\Vert _{H}^{2}-\gamma ^{-2}\left\Vert
B_{1}^{\ast }Py(t)\right\Vert _{H}^{2}+2(w(t),B_{1}^{\ast }Py(t))_{W}\right)
dt,
\end{equation*}%
since $y(0)=0$ and $\lim_{t\rightarrow \infty }(Py(t),y(t))_{H}=0.$
Therefore,%
\begin{eqnarray*}
&&\int_{0}^{\infty }\left( \left\Vert C_{1}y(t)\right\Vert
_{H}^{2}+\left\Vert B_{2}^{\ast }Py(t)\right\Vert _{H}^{2}\right) dt \\
&=&\int_{0}^{\infty }\left( -\gamma ^{-2}\left\Vert B_{1}^{\ast
}Py(t)\right\Vert _{H}^{2}+2(w(t),B_{1}^{\ast }Py(t))_{H}-\gamma
^{2}\left\Vert w(t)\right\Vert _{W}^{2}\right) dt+\int_{0}^{\infty }\gamma
^{2}\left\Vert w(t)\right\Vert _{W}^{2}dt \\
&=&\int_{0}^{\infty }\gamma ^{2}\left\Vert w(t)\right\Vert
_{W}^{2}dt-\int_{0}^{\infty }\gamma ^{2}\left\Vert \widetilde{w}%
(t)\right\Vert _{W}^{2}dt,
\end{eqnarray*}%
where 
\begin{equation}
\widetilde{w}(t)=w(t)-\gamma ^{-2}B_{1}^{\ast }Py(t).  \label{226}
\end{equation}%
If we prove that there exists $\alpha >0$ such that%
\begin{equation}
\left\Vert \widetilde{w}\right\Vert _{L^{2}(0,\infty ;W)}\geq \alpha
\left\Vert w\right\Vert _{L^{2}(0,\infty ;W)},\mbox{ }\forall w\in L^{2}(%
\mathbb{R}_{+};W),  \label{227}
\end{equation}%
it follows that%
\begin{equation*}
\gamma ^{2}\left( \left\Vert w\right\Vert _{L^{2}(\mathbb{R}%
_{+};W)}^{2}-\left\Vert \widetilde{w}\right\Vert _{L^{2}(\mathbb{R}%
_{+};W)}^{2}\right) \leq \gamma ^{2}(1-\alpha )\left\Vert w\right\Vert
_{L^{2}(\mathbb{R}_{+};W)}^{2}
\end{equation*}%
and therefore%
\begin{equation*}
\int_{0}^{\infty }\left( \left\Vert C_{1}y(t)\right\Vert _{H}^{2}+\left\Vert
B_{2}^{\ast }Py(t)\right\Vert _{H}^{2}\right) dt\leq (\gamma ^{2}-\delta
)\left\Vert w\right\Vert _{L^{2}(\mathbb{R}_{+};W)}^{2}
\end{equation*}%
with $\delta >0$ independent on $w.$ Therefore,%
\begin{equation*}
\int_{0}^{\infty }\left( \left\Vert C_{1}y(t)\right\Vert _{H}^{2}+\left\Vert
B_{2}^{\ast }Py(t)\right\Vert _{H}^{2}\right) dt\leq (\gamma ^{2}-\delta
)\int_{0}^{\infty }\left\Vert w(t)\right\Vert _{W}^{2}dt,
\end{equation*}%
which by (\ref{8-2}) implies (\ref{14-1}). We note that once (\ref{227})
proved, $\alpha $ can be chosen smaller such that $\alpha <1.$ It remains to
prove (\ref{227}) and this will be done in Lemma 3.6 given at the end of
this section.

Therefore, $G$ corresponding to $\widetilde{F}:=-B_{2}^{\ast }P$ has the
property $\left\Vert G_{\widetilde{F}}w\right\Vert _{L^{2}(\mathbb{R}%
_{+};Z)}<\gamma \left\Vert w\right\Vert _{L^{2}(\mathbb{R}_{+};W)},$ that is 
$\widetilde{F}$ is the feedback operator which solves the $H^{\infty }$%
-control problem. This ends the proof of the the first part of Theorem \ref%
{Th-main}.

Assume now that $P$ is a solution to equation (\ref{15}), satisfying (\ref%
{14-2}), such that $\Lambda _{P}=A-(B_{2}B_{2}^{\ast }-\gamma
^{-2}B_{1}B_{1}^{\ast })P$ generates an exponentially stable semigroup on $%
\mathcal{X}$. We set $y^{\ast }(t)=e^{\Lambda _{P}t}y_{0},$ for $y_{0}\in H,$
so that $y^{\ast }\in C([0,\infty );H)\cap L^{2}(\mathbb{R}_{+};H).$ Let us
define $p(t)=-Py^{\ast }(t),$ for $t\geq 0.$ Then, $p\in L^{2}(\mathbb{R}%
_{+};H)\cap C([0,\infty );H)$ and by replacing $Py^{\ast }(t)$ in (\ref{15})
we get that $p$ satisfies equation (\ref{217}) with the regularity obtained
in Lemma 3.4. The , as before. Finally, we show that the operator $\Lambda
_{P}^{1}$ generates an exponentially stable semigroup and that the
controller $\widetilde{F}y=-B_{2}^{\ast }Py$ stabilizes equation $y^{\prime
}(t)=(A+B_{2}\widetilde{F})y(t)+B_{1}w(t),$ $y(0)=0,$ arguing as before
beginning from (\ref{225-2}). This ends the proof of Theorem \ref{Th-main}.
\end{proof}

It remains to prove (\ref{227}). We set 
\begin{equation}
\Phi (w)=\left\Vert \widetilde{w}\right\Vert _{L^{2}(\mathbb{R}_{+};W)}^{2}.
\label{228}
\end{equation}

\begin{lemma}
We have 
\begin{equation}
\Phi (w)\geq \alpha \left\Vert w\right\Vert _{L^{2}(0,\infty ;W)}^{2},\mbox{
for all }w\in L^{2}(\mathbb{R}_{+};W),  \label{229}
\end{equation}%
where $\alpha >0.$
\end{lemma}

\begin{proof}
We proceed by reduction to absurdity. Assume that (\ref{229}) does not hold
and argue from this a contradiction. Thus, let $(w_{n})_{n}\subset L^{2}(%
\mathbb{R}_{+};W)$ be such that $\left\Vert w_{n}\right\Vert _{L^{2}(\mathbb{%
R}_{+};W)}=1,$ $\forall n\in \mathbb{N}$ and $\Phi (w_{n})\rightarrow 0$ as $%
n\rightarrow \infty .$ Hence, by (\ref{228}) and eq. (\ref{225-2}) we have 
\begin{equation}
\Phi (w_{n})=\left\Vert w_{n}-\gamma ^{-2}B_{1}^{\ast
}P\int_{0}^{t}e^{(A-B_{2}B_{2}^{\ast }P)t}B_{1}w_{n}(s)ds\right\Vert _{L^{2}(%
\mathbb{R}_{+};W)}\rightarrow 0,\mbox{ as }n\rightarrow \infty .  \label{230}
\end{equation}%
On the other hand, on a subsequence, we have $w_{n}\rightarrow \overline{w}$
weakly in $L^{2}(\mathbb{R}_{+};W),$ and since $\Phi $ is weakly lower
semicontinuous in $L^{2}(\mathbb{R}_{+};W)$ (because it is continuous and
convex) we have by (\ref{230}) that $\Phi (\overline{w})=0$ which implies
that 
\begin{equation*}
\overline{w}(t)=\gamma ^{-2}B_{1}^{\ast }P\int_{0}^{t}e^{(A-B_{2}B_{2}^{\ast
}P)t}B_{1}\overline{w}(s)ds,\mbox{ }\forall t\geq 0.
\end{equation*}%
By Gronwall's lemma we deduce that $\overline{w}(t)=0.$ Now, if we prove
that 
\begin{equation*}
w_{n}\rightarrow \overline{w}\mbox{ strongly in }L^{2}(\mathbb{R}_{+};W)%
\mbox{ as }n\rightarrow \infty
\end{equation*}%
(namely, that $(w_{n})_{n}$ is compact in $L^{2}(\mathbb{R}_{+};W))$ we
arrive to a contradiction because, the choice $\left\Vert w_{n}\right\Vert
_{L^{2}(\mathbb{R}_{+};W)}=1$ implies $\left\Vert \overline{w}\right\Vert
_{L^{2}(\mathbb{R}_{+};W)}=1,$ which was found before to be $0.$

To prove that $(w_{n})_{n}$ is compact in $L^{2}(\mathbb{R}_{+};W),$ by (\ref%
{230}) it suffices to show that the sequence%
\begin{equation*}
z_{n}(t)=\gamma ^{-2}B_{1}^{\ast }P\int_{0}^{t}e^{(A-B_{2}B_{2}^{\ast
}P)t}B_{1}w_{n}(s)ds,\mbox{ }t\geq 0
\end{equation*}%
is compact in $L^{2}(\mathbb{R}_{+};W),$ that is, it contains a convergent
subsequence. Taking into account that 
\begin{equation}
\left\Vert z_{n}\right\Vert _{L^{2}(T,\infty ;W)}\rightarrow 0\mbox{ as }%
T\rightarrow \infty ,\mbox{ uniformly in }n,  \label{230-0}
\end{equation}%
since $A-B_{2}B_{2}^{\ast }P$ generates an exponentially stable semigroup,
it suffices to prove that $(z_{n})_{n}$ is compact in $L^{2}(0,T;W),$ for
each $T>0.$ We set 
\begin{equation}
S(t)=e^{(A-B_{2}B_{2}^{\ast }P)t},\mbox{ }t\geq 0  \label{230-1}
\end{equation}%
and prove that $\{S(t)\}$ is compact for each $t>0.$ This means that the set~%
$\{S(t)y_{0};$ $y_{0}\in H;$ $\left\Vert y_{0}\right\Vert _{H}\leq M\}$ is
relatively compact in $H.$ Since $A-B_{2}B_{2}^{\ast }P=\Lambda _{P}-\gamma
^{-2}B_{1}B_{1}^{\ast }P$ and $B_{1}B_{1}^{\ast }P\in L(H,H)$ and $\Lambda
_{P}=A_{P}$ on $D(A_{P})$ it suffices to show that $T_{P}(t)=e^{A_{P}t}$ is
compact for each $t>0.$ This follows by density by showing first that $%
\{T_{P}(t)y_{0};$ $y_{0}\in D(A_{P}),$ $\left\Vert A_{P}y_{0}\right\Vert
_{H}+\left\Vert y_{0}\right\Vert _{H}\leq M\}$ is relatively compact in $H.$
To this end, for $\varepsilon >0,$ we write $T_{P}(t)y_{0}$ in the following
form 
\begin{eqnarray}
&&T_{P}(t)y_{0}=e^{At}y_{0}-\int_{0}^{t}e^{A(t-s)}(B_{2}B_{2}^{\ast
}Py(s)-\gamma ^{-2}B_{1}B_{1}^{\ast }Py(s))ds  \label{231} \\
&=&e^{At}y_{0}-e^{A\varepsilon }\int_{0}^{t-\varepsilon
}e^{A(t-s-\varepsilon )}(B_{2}B_{2}^{\ast }Py(s)-\gamma
^{-2}B_{1}B_{1}^{\ast }Py(s))ds  \notag \\
&&-\int_{t-\varepsilon }^{t}e^{A(t-s)}(B_{2}B_{2}^{\ast }Py(s)-\gamma
^{-2}B_{1}B_{1}^{\ast }Py(s))ds,  \notag
\end{eqnarray}%
where $y(t)=T_{P}(t)y_{0}.$ If $\mathcal{M}=\{y_{0}\in D(A_{P});$ $%
\left\Vert A_{P}y_{0}\right\Vert _{H}+\left\Vert y_{0}\right\Vert _{H}\leq
M\},$ relation (\ref{231}) yields 
\begin{eqnarray*}
&&T_{P}(t)\mathcal{M=}\left\{ e^{At}y_{0};\mbox{ }y_{0}\in \mathcal{M}%
\right\}  \\
&&-\left\{ e^{A\varepsilon }\int_{0}^{t-\varepsilon }e^{A(t-s-\varepsilon
)}(B_{2}B_{2}^{\ast }Py(s)-\gamma ^{-2}B_{1}B_{1}^{\ast }Py(s))ds;\mbox{ }%
y_{0}\in \mathcal{M}\right\}  \\
&&-\left\{ \int_{t-\varepsilon }^{t}e^{A(t-s)}(B_{2}B_{2}^{\ast
}Py(s)-\gamma ^{-2}B_{1}B_{1}^{\ast }Py(s))ds;\mbox{ }y_{0}\in \mathcal{M}%
\right\} =\mathcal{M}_{1}+\mathcal{M}_{2}+\mathcal{M}_{3}.
\end{eqnarray*}%
In the sum above, $\mathcal{M}_{1}$ is relatively compact because $e^{At}$
is compact by $(i_{1})$. Next, we write $\mathcal{M}_{2}=\mathcal{M}_{21}+%
\mathcal{M}_{22}$ where $\mathcal{M}_{2i}=\left\{ e^{A\varepsilon
}\int_{0}^{t-\varepsilon }e^{A(t-s-\varepsilon )}B_{i}B_{i}^{\ast }Py(s)ds;%
\mbox{ }y_{0}\in \mathcal{M}\right\} ,$ $i=1,2.$ $\mathcal{M}_{21}$ is
relatively compact because $e^{A\varepsilon }$ is compact and $%
\int_{0}^{t-\varepsilon }e^{A(t-s-\varepsilon )}B_{1}B_{1}^{\ast }Py(s)ds$
is bounded, 
\begin{eqnarray*}
\left\Vert \int_{0}^{t-\varepsilon }e^{A(t-s-\varepsilon )}B_{1}B_{1}^{\ast
}Py(s)ds\right\Vert _{H} &\leq &C\int_{0}^{t}\left\Vert B_{1}B_{1}^{\ast
}Py(s)\right\Vert _{U}ds \\
&\leq &C\int_{0}^{t}\left\Vert B_{1}^{\ast }Py(s)\right\Vert _{H}ds\leq
Ct\left\Vert y_{0}\right\Vert _{H}.
\end{eqnarray*}%
Then, 
\begin{eqnarray*}
&&\left\Vert \int_{0}^{t-\varepsilon }e^{A(t-s-\varepsilon
)}B_{2}B_{2}^{\ast }Py(s)ds\right\Vert _{H}=\sup_{\varphi \in H,\left\Vert
\varphi \right\Vert _{H}\leq 1}\left( \int_{0}^{t-\varepsilon
}e^{A(t-s-\varepsilon )}B_{2}B_{2}^{\ast }Py(s)ds,\varphi \right) _{H} \\
&\leq &\sup_{\left\Vert \varphi \right\Vert _{H}\leq
1}\int_{0}^{t-\varepsilon }\left( B_{2}^{\ast }Py(s),B_{2}^{\ast }e^{A^{\ast
}(t-s-\varepsilon )}\varphi \right) _{U}ds\leq \sup_{\left\Vert \varphi
\right\Vert _{H}\leq 1}\int_{0}^{t-\varepsilon }\left\Vert B_{2}^{\ast
}Py(s)\right\Vert _{U}\left\Vert B_{2}^{\ast }e^{A^{\ast }(t-s-\varepsilon
)}\varphi \right\Vert _{U}ds \\
&\leq &\sup_{\left\Vert \varphi \right\Vert _{H}\leq
1}\int_{0}^{t-\varepsilon }\left\Vert Py(s)\right\Vert _{D(A^{\ast
})}\left\Vert B_{2}^{\ast }e^{A^{\ast }(t-s-\varepsilon )}\right\Vert
_{L(H,U)}\left\Vert \varphi \right\Vert _{H}ds \\
&\leq &\int_{0}^{t-\varepsilon }\left\Vert y(s)\right\Vert
_{D(A_{P})}\left\Vert B_{2}^{\ast }e^{A^{\ast }(t-s-\varepsilon
)}\right\Vert _{L(H,U)}ds \\
&\leq &\int_{0}^{t-\varepsilon }\left\Vert A_{P}y(s)\right\Vert
_{H}\left\Vert B_{2}^{\ast }e^{A^{\ast }(t-s-\varepsilon )}\right\Vert
_{L(H,U)}ds\leq C\left\Vert A_{P}y_{0}\right\Vert
_{H}\int_{0}^{t-\varepsilon }\left\Vert B_{2}^{\ast }e^{A^{\ast
}(t-s-\varepsilon )}\right\Vert _{L(H,U)}ds\leq C_{T},
\end{eqnarray*}%
hence $\mathcal{M}_{22}$ is relatively compact, too. We also have 
\begin{equation*}
\left\Vert \int_{t-\varepsilon }^{t}e^{A(t-s)}B_{1}B_{1}^{\ast
}Py(s)ds\right\Vert _{H}\leq C\int_{t-\varepsilon }^{t}\left\Vert
B_{2}B_{2}^{\ast }Py(s)\right\Vert _{H}ds\leq C\varepsilon 
\end{equation*}%
and similarly we estimate that the term corresponding to $B_{2}B_{2}^{\ast }P
$ is bounded by $C\varepsilon .$ Since $\varepsilon $ is arbitrary it
follows that $T_{P}(t)\mathcal{M}$ is compact and, as mentioned earlier, it
follows by density that the set $\left\{ T_{P}(t)y_{0};\mbox{ }\left\Vert
y_{0}\right\Vert _{H}\leq M\right\} $ is compact for each $M$ and $t>0,$
fixed. Now, coming back to $z_{n}$ we write 
\begin{equation*}
z_{n}(t)=\gamma ^{-2}B_{1}^{\ast }PS(\varepsilon )\left(
\int_{0}^{t-\varepsilon }S(t-s-\varepsilon )B_{1}w_{n}(s)ds\right) +\gamma
^{-2}B_{1}^{\ast }P\int_{t-\varepsilon }^{t}S(t-s)B_{1}w_{n}(s)ds
\end{equation*}%
and get 
\begin{eqnarray*}
&&\left\Vert \int_{0}^{t-\varepsilon }S(t-s-\varepsilon
)B_{1}w_{n}(s)ds\right\Vert _{H}\leq C\left\Vert \int_{0}^{t-\varepsilon
}e^{-\beta (t-s-\varepsilon )}B_{1}w(s)ds\right\Vert _{H} \\
&\leq &C\int_{0}^{t-\varepsilon }\left\Vert B_{1}w_{n}(s)\right\Vert
_{H}ds\leq C\left\Vert w\right\Vert _{L^{2}(\mathbb{R}_{+};W)}\leq C,\mbox{ }%
\forall t\geq 0,
\end{eqnarray*}%
hence, $\left\{ S(\varepsilon )\left( \int_{0}^{t-\varepsilon
}S(t-s-\varepsilon )B_{1}w_{n}(s)ds\right) \right\} $ is compact in $H.$

Taking into account that $\left\Vert \int_{t-\varepsilon
}^{t}S(t-s)B_{1}w(s)ds\right\Vert _{H}\leq C\varepsilon ,$ it follows that $%
(z_{n}(t))_{n}$ is compact in $H,$ for every $t>0.$ Also, it is
equi-uniformly continuous, that is $\left\Vert
z_{n}(t+h)-z_{n}(t)\right\Vert _{H}\leq \varepsilon $ if $\left\vert
h\right\vert \leq \delta (\varepsilon ),$ for any $t.$ The latter follows
because the semigroup $S(t)$ is continuous for $t>0$ in the uniform operator
topology (see \cite{Pazy}, p. 48, Theorem 3.2), and this means that $%
\left\Vert (S(t+h)-S(t))\theta \right\Vert _{H}\leq \delta _{1}(h)\left\Vert
\theta \right\Vert _{H},$ where $\delta _{1}(h)\rightarrow 0,$ and $\theta
\in H.$ Then, 
\begin{eqnarray*}
&&\left\Vert z_{n}(t+h)-z_{n}(t)\right\Vert _{H}\leq
C_{1}\int_{t}^{t+h}\left\Vert S(t+h-s)B_{1}Pw_{n}(s)\right\Vert _{H}ds \\
&&+C_{2}\int_{0}^{t}\left\Vert (S(t+h-s)-S(t-s))B_{1}Pw_{n}(s)\right\Vert
_{H}ds \\
&\leq &C_{1}\int_{t}^{t+h}e^{-\beta (t+h-s)}\left\Vert w_{n}(s)\right\Vert
_{H}ds+C_{2}\delta _{2}(h),
\end{eqnarray*}%
where $\delta _{2}(h)\rightarrow 0$ as $h\rightarrow 0.$ Then, by
Ascoli-Arzel\`{a}'s theorem, $(z_{n})_{n}$ is compact in $C([0,T];H),$ for
every $T>0$ and so $z_{n}\rightarrow z$ strongly in $L^{2}(0,T;H),$ for
every $T>0.$ Recalling (\ref{230-0}) we note that 
\begin{equation*}
\left\Vert z_{n}-z\right\Vert _{L^{2}(\mathbb{R}_{+};H)}^{2}=\int_{0}^{T}%
\left\Vert z_{n}(t)-z(t)\right\Vert _{H}^{2}dt+\int_{T}^{\infty }\left\Vert
z_{n}(t)-z(t)\right\Vert _{H}^{2}dt\rightarrow 0,\mbox{ as }n\rightarrow
\infty 
\end{equation*}%
because the first term tends to $0$ by the compactness argument developed
before and 
\begin{equation*}
\int_{T}^{\infty }\left\Vert z_{n}(t)-z(t)\right\Vert _{H}^{2}dt\leq
2\int_{T}^{\infty }\left\Vert z_{n}(t)\right\Vert
_{H}^{2}dt+\int_{T}^{\infty }\left\Vert z(t)\right\Vert
_{H}^{2}dt\rightarrow 0
\end{equation*}%
by (\ref{230-0}) and the fact that $A-B_{2}B_{2}^{\ast }P$ generates an
exponentially stable semigroup.

Going back to (\ref{230}), it follows that $\left\Vert
(w_{n}-z_{n})(t)\right\Vert _{W}\rightarrow 0,$ a.e. $t>0,$ and so $%
(w_{n})_{n}$ is compact, as claimed. This ends the proof of Lemma 3.5 and
also of Theorem 3.1.
\end{proof}

\begin{remark}
Theorem \ref{Th-main} reduces the existence of a robust feedback controller $%
F$ satisfying (\ref{14-1}) to the existence of a solution $P$ to (\ref{15})
in the same way as for $B_{1}=B_{2}=D_{1}=0,$ $C_{1}=I,$ the Lyapunov
equation $A^{\ast }P+PA=I$ is related to the stability of the semigroup $%
e^{At}.$ In the specific examples discussed in the next sections we shall
show that the operatorial equation (\ref{15}) reduces to a nonlinear
integro-differential elliptic equation.
\end{remark}

\section{The case of a $N$-$D$ distributed control\label{Distributed}}

\setcounter{equation}{0}

Let $\Omega $ be an open bounded subset of $\mathbb{R}^{N},$ $N>3$ with the
boundary $\Gamma =\partial \Omega $ sufficiently smooth and assume that $%
0\in \Omega .$ We consider the following singular system%
\begin{eqnarray}
&&\left. y_{t}-\Delta y-\frac{\lambda y}{|x|^{2}}-a(x)y=B_{1}w+B_{2}u,%
\right. \left. \mbox{in }(0,\infty )\times \Omega ,\right.  \label{21} \\
&&\left. y=0,\right. \left. \mbox{ \ \ \ \ \ \ \ \ \ \ \ \ \ \ \ \ \ \ \ \ \
\ \ \ \ \ \ \ \ \ \ \ \ \ \ \ \ \ \ \ \ on }(0,\infty )\times \Gamma ,\right.
\label{21-1} \\
&&\left. y(0)=y_{0},\right. \left. \mbox{ \ \ \ \ \ \ \ \ \ \ \ \ \ \ \ \ \
\ \ \ \ \ \ \ \ \ \ \ \ \ \ \ \ \ \ \ in }\Omega ,\right.  \label{21-2} \\
&&\left. z=C_{1}y+D_{1}u,\right. \left. \mbox{ \ \ \ \ \ \ \ \ \ \ \ \ \ \ \
\ \ \ \ \ \ \ \ \ \ \ \ \ \ in }(0,\infty )\times \Omega ,\right.
\label{21-3}
\end{eqnarray}%
where $\lambda >0,$ $\left\vert \cdot \right\vert $ denotes the Euclidian
norm in $\mathbb{R}^{N},$ for any $N=1,2,...$, according the case and $a$
has the expression 
\begin{equation}
a(x)=a_{0}\chi _{\Omega _{0}}(x),\mbox{ }a_{0}>0,\mbox{ }\Omega _{0}\subset
\Omega .  \label{21-4}
\end{equation}%
In this problem 
\begin{equation}
y_{0}\in L^{2}(\Omega )  \label{21-5}
\end{equation}%
and we choose%
\begin{equation}
H=W=Z=L^{2}(\Omega ),\mbox{ }U=\mathbb{R},  \label{21-6}
\end{equation}%
\begin{eqnarray}
B_{1}w &=&\chi _{\omega _{1}}(x)w,\mbox{ \ }B_{2}u=b(x)u,\mbox{ \ }
\label{22} \\
C_{1}y &=&\chi _{\Omega _{C}}(x)y,\mbox{ \ }D_{1}u=d(x)u,\mbox{ }x\in \Omega
,\mbox{ }u\in \mathbb{R},  \notag
\end{eqnarray}%
where $\Omega _{0},$ $\Omega _{C},$ $\omega _{1}$ are open sets of $\Omega ,$
$\chi _{\omega }$ is characteristic functions of the set $\omega \subset
\Omega ,$ 
\begin{equation}
\omega _{1}\sqsubseteq \Omega ,\mbox{ }\Omega _{0}\sqsubseteq \Omega
_{C}\subset \Omega ,\mbox{ }  \label{22-5}
\end{equation}%
and 
\begin{equation}
b\in L^{2}(\Omega ),\mbox{ }d\in L^{2}(\Omega ),\mbox{ }d(x)=\chi _{\Omega
\backslash \Omega _{C}.}.  \label{22-1}
\end{equation}%
We begin by checking the hypotheses $(i_{1})-(i_{4}).$

\medskip

$(i_{1})$ By their expressions we see that 
\begin{equation*}
B_{1},C_{1}\in L(L^{2}(\Omega ),L^{2}(\Omega )),\mbox{ }B_{2},D_{1}\in L(%
\mathbb{R},L^{2}(\Omega ))
\end{equation*}%
and $B_{2}^{\ast }:L^{2}(\Omega )\rightarrow \mathbb{R}$ is defined by 
\begin{equation}
B_{2}^{\ast }v=\int_{\Omega }b(x)v(x)dx,\mbox{ for }v\in L^{2}(\Omega ).
\label{22-3}
\end{equation}%
We recall the Hardy inequality (\ref{HN}) and consider $\lambda <H_{N}.$ We
introduce the self-adjoint operator 
\begin{equation}
A:D(A)\subset L^{2}(\Omega )\rightarrow L^{2}(\Omega ),\mbox{ }Ay=\Delta y+%
\frac{\lambda y}{|x|^{2}}+ay,  \label{23}
\end{equation}%
with%
\begin{equation}
D(A)=\{y\in H_{0}^{1}(\Omega );\mbox{ }Ay\in L^{2}(\Omega )\}.  \label{24}
\end{equation}%
It is clear that $\overline{D(A)}=L^{2}(\Omega )$ because $D(A)$ contains $%
C_{0}^{\infty }(\Omega \backslash \{0\}).$ Then, equation (\ref{21}) can be
equivalently written%
\begin{equation}
y^{\prime }(t)=Ay(t)+B_{1}w(t)+B_{2}u(t),\mbox{ }t\geq 0.  \label{24-1}
\end{equation}%
In order to show that $A$ generates a $C_{0}$-semigroup on $L^{2}(\Omega ),$
we have to prove that $A$ is $\omega $-$m$-dissipative on $L^{2}(\Omega ),$
or that $-A$ is $\omega $-$m$-accretive on $L^{2}(\Omega )$ (see \cite%
{VB-book-2010}, p. 155).

\begin{lemma}
Let $\lambda <H_{N}.$ The operator $-A$ is $\omega $-m-accretive on $%
L^{2}(\Omega ),$ for $\omega >a_{0}.$
\end{lemma}

\begin{proof}
This means to show that $-A$ is $\omega $-accretive, that is $((\omega
I-A)y,y)_{2}\geq 0$ for some $\omega >0$ and all $y\in L^{2}(\Omega )$ and
that $\omega I-A$ is surjective. To this end we shall use several times the
Hardy inequality (\ref{HN}) which ensures that $\frac{y}{x}\in L^{2}(\Omega
) $ if $y\in H_{0}^{1}(\Omega ).$ We have 
\begin{eqnarray*}
((\omega I-A)y,y)_{2} &=&\omega \int_{\Omega }\left\vert y\right\vert
^{2}dx+\int_{\Omega }\left\vert \nabla y\right\vert ^{2}dx-\lambda
\int_{\Omega }\frac{\left\vert y\right\vert ^{2}}{\left\vert x\right\vert
^{2}}dx-a_{0}\int_{\Omega _{0}}\left\vert y\right\vert ^{2}dx \\
&\geq &\left( 1-\frac{\lambda }{H_{N}}\right) \int_{\Omega }\left\vert
\nabla y\right\vert ^{2}dx+(\omega -a_{0})\int_{\Omega }\left\vert
y\right\vert ^{2}dx \\
&\geq &\frac{1}{2}\left( 1-\frac{\lambda }{H_{N}}\right) \left\Vert \nabla
y\right\Vert _{2}^{2}+\frac{H_{N}}{2}\left( 1-\frac{\lambda }{H_{N}}\right)
\left\Vert \frac{y}{x}\right\Vert _{2}^{2}+(\omega -a_{0})\left\Vert
y\right\Vert _{2}^{2}
\end{eqnarray*}%
which shows that $-A$ is $\omega $-accretive on $L^{2}(\Omega )$ for $%
\lambda <H_{N}$ and $\omega >a_{0}.$

To prove the surjectivity of $\omega I-A,$ we show that the range $R(\omega
I-A)=L^{2}(\Omega ).$ Thus, let $f\in L^{2}(\Omega )$ and prove that the
equation 
\begin{equation}
\omega y-Ay=f  \label{24-3}
\end{equation}%
has a solution $y\in D(A)$, by the equivalent variational formulation
expressed by the minimization problem 
\begin{equation}
\min_{y\in H_{0}^{1}(\Omega )}\left\{ J(y)=\int_{\Omega }\left( \frac{1}{2}%
\left\vert \nabla y\right\vert ^{2}-\frac{\lambda }{2}\frac{y^{2}}{%
\left\vert x\right\vert ^{2}}-\frac{\omega -a(x)}{2}y^{2}-fy\right)
dx\right\} ,  \label{24-0}
\end{equation}%
subject to (\ref{24-1}) and $y(0)=y_{0}\in L^{2}(\Omega ).$ For $\omega
>a_{0}$ we have 
\begin{equation*}
\frac{1}{2}\left( 1-\frac{\lambda }{H_{N}}\right) \int_{\Omega }\left\vert
\nabla y\right\vert ^{2}dx+(\omega -a_{0})\int_{\Omega }y^{2}dx-\frac{1}{%
2(\omega -a_{0})}\int_{\Omega }\left\vert f\right\vert ^{2}dx\leq J(\varphi
)<\infty
\end{equation*}%
so that $J$ has an infimum $d.$ Taking a minimizing sequence $(y_{n})_{n}$
we have 
\begin{equation}
d\leq J(y_{n})\leq d+\frac{1}{n}  \label{24-2}
\end{equation}%
and so 
\begin{equation*}
\left\Vert \nabla y_{n}\right\Vert _{2}+\left\Vert y_{n}\right\Vert
_{2}+\left\Vert \frac{y_{n}}{x}\right\Vert _{2}\leq C_{N}\mbox{ for }\omega
>a_{0}.
\end{equation*}%
Further, $C,$ $C_{N},$ $C_{T}$ denote some constants (which may change from
line to line), $C_{N}$ depending on $N,$ via $\lambda <H_{N}$ and $C_{T}$
depending on $T.$

We deduce that on a subsequence denoted still by $n$ it follows that 
\begin{equation*}
y_{n}\rightarrow y\mbox{ weakly in }H_{0}^{1}(\Omega ),\mbox{ }\frac{y_{n}}{x%
}\rightarrow l\mbox{ weakly in }L^{2}(\Omega )
\end{equation*}%
and by compactness $y_{n}\rightarrow y$ strongly in $L^{2}(\Omega ).$ Then $%
\frac{y_{n}}{x}\rightarrow \frac{y}{x}$ a.e. on $\Omega $ and $l=\frac{y}{x}$
by the Vitali's theorem. We can now pass to the limit in (\ref{24-2}),
relying on the weakly lower semicontinuity of $J$ and get that $J(y)=d,$
that is $y$ realizes the minimum in (\ref{24-0}).

Next, we give a variation $y^{\sigma }=y+\sigma \eta ,$ for $\sigma >0$ and $%
\eta \in H_{0}^{1}(\Omega ),$ and particularize the condition of optimality,
namely $J(\widetilde{y})\geq J(y)$ for any $\widetilde{y}\in
H_{0}^{1}(\Omega )$ for $\widetilde{y}=y^{\sigma }.$ We calculate 
\begin{equation*}
\lim_{\sigma \rightarrow 0}\frac{J(y^{\sigma })-J(y)}{\sigma }=\int_{\Omega
}\left( (\omega -a(x))y\eta +\nabla y\cdot \nabla \eta -\frac{\lambda y\eta 
}{\left\vert x\right\vert ^{2}}-f\eta \right) dx\geq 0.
\end{equation*}%
Repeating the calculus for $\sigma \rightarrow -\sigma $ we get the reverse
inequality, so that finally we can write%
\begin{equation*}
\int_{\Omega }\left\langle (\omega -a(x))y-\Delta y-\frac{\lambda y}{%
\left\vert x\right\vert ^{2}}-f,\eta \right\rangle _{H^{-1}(\Omega
),H_{0}^{1}(\Omega )}dx=0\mbox{ for all }\eta \in H_{0}^{1}(\Omega ),
\end{equation*}%
which implies that $y$ is the weak solution to the equation (\ref{24-3}).
The solution is also unique because $J$ is strictly convex and the system is
linear. By (\ref{24-3}) we see that $Ay\in L^{2}(\Omega ),$ so that $y\in
D(A).$
\end{proof}

In conclusion, $A$ generates an analytic $C_{0}$-semigroup on $L^{2}(\Omega
) $ for $\lambda <H_{N}.$

Moreover, since as earlier seen, the operator $(\omega I-A)^{-1}$ is a
compact operator for $\omega >a_{0},$ it follows that $e^{At}$ is compact
for all $t>0.$

\medskip

$(i_{2})$ Let $y_{0}\in L^{2}(\Omega ),$ $u\in L^{2}(\mathbb{R}_{+},\mathbb{R%
}),$ $w\in L^{2}(\mathbb{R}_{+};L^{2}(\Omega )).$ Since $B_{1}w+B_{2}u\in
L^{2}(0,T;L^{2}(\Omega ))$ and $y_{0}\in \overline{D(A)}=L^{2}(\Omega ),$
eq. (\ref{24-1}) with $y(0)=y_{0}$ has a unique mild solution $y\in
C([0,T],L^{2}(\Omega )),$ given by (\ref{13}) for any $T>0$ (see \cite%
{VB-book-2010}, p. 131, Corollary 4.1). The solution also satisfies $y\in
L^{2}(0,T;H_{0}^{1}(\Omega ))\cup W^{1,2}(0,T;H^{-1}(\Omega )).$

\medskip

In order to prove $(i_{3})$ we provide the following lemma.

\begin{lemma}
Let $\lambda <H_{N}.$ Then, the pair $(A,C_{1})$ is exponentially detectable.
\end{lemma}

\begin{proof}
Let $K\equiv -kI$, with $k\geq a_{0}$ and set $A_{1}=A+KC_{1}.$ This is
still $\omega $-$m$-accretive, so that $A_{1}$ generates a $C_{0}$-semigroup
on $L^{2}(\Omega ),$ $S_{1}(t)=e^{A_{1}t}.$ Hence $y(t)=e^{A_{1}t}y_{0}$
satisfies%
\begin{equation}
\frac{dy}{dt}(t)=A_{1}y(t),\mbox{ }t\geq 0,\mbox{ }y(0)=y_{0}.  \label{29}
\end{equation}%
Recalling the expression of $C_{1},$ multiplying (\ref{29}) by $y(t)$ and
applying again (\ref{HN}) we get 
\begin{equation}
\frac{1}{2}\frac{d}{dt}\left\Vert y(t)\right\Vert _{2}^{2}+\left( 1-\frac{%
\lambda }{H_{N}}\right) \left\Vert \nabla y(t)\right\Vert
_{2}^{2}+k\int_{\Omega _{C}}\left\vert y(t)\right\vert ^{2}ds\leq
a_{0}\int_{\Omega _{0}}\left\vert y(t)\right\vert ^{2}dx.  \label{30}
\end{equation}%
We take into account that $\Omega _{0}\sqsubseteq \Omega _{C}$ and $k\geq
a_{0}$, and integrate from $0$ to $t.$ We obtain 
\begin{equation*}
\frac{1}{2}\left\Vert y(t)\right\Vert _{2}^{2}+\left( 1-\frac{\lambda }{H_{N}%
}\right) \int_{0}^{t}\left\Vert \nabla y(s)\right\Vert
_{2}^{2}ds+(k-a_{0})\int_{\Omega _{0}}\left\vert y(t)\right\vert ^{2}ds\leq 
\frac{1}{2}\left\Vert y_{0}\right\Vert _{2}^{2},\mbox{ }\forall t>0.
\end{equation*}%
From here and the Poincar\'{e} inequality it follows that 
\begin{equation}
\int_{0}^{t}\left\Vert y(s)\right\Vert _{2}^{2}ds\leq C_{N}\left\Vert
y_{0}\right\Vert _{2}^{2},\mbox{ for all }t>0,  \label{31}
\end{equation}%
with $C_{N}$ a constant depending on $H_{N}.$ Letting $t\rightarrow \infty $
in (\ref{31}) we finally get that%
\begin{equation}
\int_{0}^{\infty }\left\Vert y(s)\right\Vert _{2}^{2}ds\leq C_{N}\left\Vert
y_{0}\right\Vert _{2}^{2}.  \label{31-0}
\end{equation}

This means by Datko's result, previously recalled, that $e^{A+KC_{1}}$
generates an exponentially stable semigroup, that is there exists $\alpha >0$
such that 
\begin{equation*}
\left\Vert e^{(A+KC_{1})t}y\right\Vert _{2}\leq Ce^{-\alpha t}y_{2}\mbox{
for all }y\in L^{2}(\Omega ).
\end{equation*}%
Then, 
\begin{eqnarray*}
&&\int_{0}^{\infty }\left\Vert B_{2}^{\ast }e^{(A^{\ast }+C_{1}^{\ast
}K^{\ast })t}y\right\Vert _{U}dt\leq C\int_{0}^{\infty }\left\Vert
e^{(A^{\ast }+C_{1}^{\ast }K^{\ast })t}y\right\Vert _{2}dt \\
&\leq &C\left\Vert y\right\Vert _{2}\int_{0}^{\infty }e^{-\alpha
t}dt=C\left\Vert y\right\Vert _{2},\mbox{ }\forall y\in L^{2}(\Omega ),
\end{eqnarray*}%
that is (\ref{12-0}) is verified.
\end{proof}

$(i_{4})$ By (\ref{22-1}) we have 
\begin{equation*}
\left\Vert D_{1}u\right\Vert _{2}^{2}=u^{2}\left\Vert d\right\Vert
_{L^{2}(\Omega \backslash \Omega _{C})}=u^{2}
\end{equation*}%
and 
\begin{equation*}
D_{1}^{\ast }C_{1}y=\int_{\Omega }d(x)\chi _{\Omega _{C}}(x)y(x)dx=0.
\end{equation*}%
The hypotheses being checked, we can formulate the $H^{\infty }$-control
problem for system (\ref{21})-(\ref{21-3}) as in Theorem \ref{Th-main}.

In order to explicit Theorem \ref{Th-main} and to give a differential
formulation for it, as announced in Remark 3.7, we recall that the linear
continuous operator $P\in L(L^{2}(\Omega ),L^{2}(\Omega ))$ can be
represented by the L. Schwartz kernel theorem (see e.g., \cite{Lions-control}%
, p. 166) as an integral operator with a kernel $P_{0}\in L^{2}(\Omega
\times \Omega ),$ namely%
\begin{equation}
P\varphi (x)=\int_{\Omega }P_{0}(x,\xi )\varphi (\xi )d\xi ,\mbox{ for all }%
\varphi \in C_{0}^{\infty }(\Omega ).  \label{36}
\end{equation}%
By (\ref{22}) and (\ref{22-3}) we have%
\begin{eqnarray}
B_{1}B_{1}^{\ast }\varphi (x) &=&\chi _{\omega _{1}}(x)\varphi (x),\mbox{ }%
C_{1}C_{1}^{\ast }\varphi (x)=\chi _{\Omega _{C}}(x)\varphi (x),\mbox{ } 
\notag \\
B_{1}B_{1}^{\ast }P\varphi (x) &=&\chi _{\omega _{1}}(x)\int_{\Omega
}P_{0}(x,\xi )\varphi (\xi )d\xi ,\mbox{ }  \notag \\
PB_{1}B_{1}^{\ast }P\varphi (x) &=&\int_{\Omega }\int_{\Omega }\chi _{\omega
_{1}}(\overline{\xi })P_{0}(x,\overline{\xi })P_{0}(\overline{\xi },\xi
)\varphi (\xi )d\overline{\xi }d\xi   \label{36-0}
\end{eqnarray}%
\begin{eqnarray}
B_{2}B_{2}^{\ast }\varphi (x) &=&b(x)\int_{\Omega }b(\overline{x})\varphi (%
\overline{x})d\overline{x},\mbox{ }x\in \Omega ,  \notag \\
B_{2}B_{2}^{\ast }P\varphi (x) &=&b(x)\int_{\Omega }\int_{\Omega }b(%
\overline{x})P_{0}(\overline{x},\xi )\varphi (\xi )d\overline{x}d\xi ,\mbox{ 
}x\in \Omega ,  \notag \\
PB_{2}B_{2}^{\ast }P\varphi (x) &=&\int_{\Omega }\int_{\Omega }\int_{\Omega
}P_{0}(x,\overline{\xi })P_{0}(\overline{x},\xi )b(\overline{\xi })b(%
\overline{x})\varphi (\xi )d\overline{x}d\overline{\xi }d\xi .  \label{36-1}
\end{eqnarray}%
Moreover, by a straightforward calculation we obtain 
\begin{equation}
A^{\ast }P\varphi (x)=\int_{\Omega }\left( \Delta _{x}P_{0}(x,\xi )+\frac{%
\lambda P_{0}(x,\xi )}{|x|^{2}}+a(x)P_{0}(x,\xi )\right) \varphi (\xi )d\xi ,
\label{36-2}
\end{equation}%
\begin{equation}
PA\varphi (x)=\int_{\Omega }\varphi (\xi )\left( \Delta _{\xi }P_{0}(x,\xi )+%
\frac{\lambda P_{0}(x,\xi )}{|\xi |^{2}}+a(\xi )P_{0}(x,\xi )\right) d\xi ,
\label{36-3}
\end{equation}%
and by denoting $E:=B_{2}B_{2}^{\ast }-\gamma ^{-2}B_{1}B_{1}^{\ast },$ we
have%
\begin{eqnarray*}
PEP\varphi (x) &=&\int_{\Omega }\varphi (\xi )d\xi \int_{\Omega
}\int_{\Omega }P_{0}(x,\overline{\xi })P_{0}(\overline{x},\xi )b(\overline{%
\xi })b(\overline{x})d\overline{x}d\overline{\xi } \\
&&-\gamma ^{-2}\int_{\Omega }\varphi (\xi )d\xi \int_{\Omega }\chi _{\omega
_{1}}(\overline{\xi })P_{0}(x,\overline{\xi })P_{0}(\overline{\xi },\xi )d%
\overline{\xi }.
\end{eqnarray*}%
For $x\in \Omega $ we define the distribution $\mu _{x}\in \mathcal{D}%
^{\prime }(\Omega )$ by 
\begin{equation*}
\mu _{x}(\varphi )=\chi _{\Omega _{C}}(x)\varphi (x)=\int_{\Omega }\delta
(x-\xi )\chi _{\Omega _{C}}(\xi )\varphi (\xi )d\xi ,\mbox{ }\forall \varphi
\in C_{0}^{\infty }(\Omega ),
\end{equation*}%
where $\delta $ is the Dirac distribution. Then, by replacing all these in (%
\ref{15}), we deduce the equation%
\begin{eqnarray}
&&\Delta _{x}P_{0}(x,\xi )+\Delta _{\xi }P_{0}(x,\xi )+\lambda P_{0}(x,\xi
)\left( \frac{1}{|x|^{2}}+\frac{1}{|\xi |^{2}}\right) +(a(x)+a(\xi
))P_{0}(x,\xi )  \notag \\
&&-\int_{\Omega }\int_{\Omega }P_{0}(x,\overline{\xi })P_{0}(\overline{x}%
,\xi )b(\overline{\xi })b(\overline{x})d\overline{x}d\overline{\xi }+\gamma
^{-2}\int_{\Omega }\chi _{\omega _{1}}(\overline{\xi })P_{0}(x,\overline{\xi 
})P_{0}(\overline{\xi },\xi )d\overline{\xi }  \label{37} \\
&=&-\delta (x-\xi )\chi _{\Omega _{C}}(\xi ),\mbox{ in }\mathcal{D}^{\prime
}(\Omega \times \Omega ).  \notag
\end{eqnarray}%
This equation is accompanied by the conditions 
\begin{equation}
P_{0}(x,\xi )=0,\mbox{ }\forall (x,\xi )\in \Gamma \times \Gamma ,
\label{38}
\end{equation}%
\begin{equation}
P_{0}(x,\xi )=P(\xi ,x),\mbox{ }\forall (x,\xi )\in \Omega \times \Omega ,
\label{39}
\end{equation}%
\begin{equation}
P_{0}(x,\xi )\geq 0,\mbox{ }\forall (x,\xi )\in \Omega \times \Omega 
\label{40}
\end{equation}%
and so we can enounce the following

\begin{theorem}
Let $\gamma >0$ and let $A,$ $B_{1},$ $B_{2},$ $C_{1}$ and $D_{1}$ be given
by (\ref{23}) and (\ref{22}), respectively. Then there exists $\widetilde{F}%
\in L(L^{2}(\Omega ),\mathbb{R})$ which solves the $H^{\infty }$-control
problem for system (\ref{21})-(\ref{21-3}) if and only if there exists a
solution $P_{0}\in D(A)\times D(A)$ to (\ref{37})-(\ref{38}), satisfying (%
\ref{39})-(\ref{40}). Moreover, in this case 
\begin{equation}
\widetilde{F}y=-\int_{\Omega }\int_{\Omega }b(x)P_{0}(x,\xi )y(\xi )d\xi dx,%
\mbox{ }\forall y\in L^{2}(\Omega ),  \label{40-0}
\end{equation}%
is a feedback controller which solves the $H^{\infty }$-problem for system (%
\ref{21})-(\ref{21-3}).
\end{theorem}

In this case it is easily seen that $\Lambda _{P}=A-B_{2}B_{2}^{\ast
}P+\gamma ^{-2}B_{1}B_{1}^{\ast }P$ has the domain $D(\Lambda _{P})=D(A),$
and since $\Lambda _{P}$ is closed it follows that $\mathcal{X}=D(A).$
Moreover, by (\ref{40-0}) we see that $\widetilde{F}\in L(L^{2}(\Omega ),%
\mathbb{R}).$

A\ direct approach of problem (\ref{37})-(\ref{40}) is an interesting
problem by itself but is beyond the objective of this work.

\section{Dirichlet boundary control\label{Boundary}}

As in the previous section let $\Omega $ be an open bounded subset of $%
\mathbb{R}^{N},$ $N>3$ with the boundary $\Gamma =\partial \Omega $
sufficiently smooth and such that $0\in \Omega .$ Consider the following
system%
\begin{eqnarray}
&&\left. y_{t}-\Delta y-\frac{\lambda y}{|x|^{2}}-a(x)y=B_{1}w,\right.
\left. \mbox{in }(0,\infty )\times \Omega ,\right.  \label{41} \\
&&\left. y=\widetilde{u},\right. \left. \mbox{ \ \ \ \ \ \ \ \ \ \ \ \ \ \ \
\ \ \ \ \ \ \ \ \ \ \ \ \ \ \ \ \ on }(0,\infty )\times \Gamma ,\right.
\label{42} \\
&&\left. y(0)=y_{0},\right. \left. \mbox{ \ \ \ \ \ \ \ \ \ \ \ \ \ \ \ \ \
\ \ \ \ \ \ \ \ \ \ in }\Omega ,\right.  \label{43} \\
&&\left. z=C_{1}y+D_{1}u,\right. \left. \mbox{ \ \ \ \ \ \ \ \ \ \ \ \ \ \ \
\ \ \ \ \ in }(0,\infty )\times \Omega ,\right.  \label{44}
\end{eqnarray}%
where $y_{0}\in L^{2}(\Omega )$, $a$ is again given by (\ref{21-4}) and 
\begin{eqnarray}
\widetilde{u}(t,x) &=&\sum\limits_{j=1}^{m}\alpha _{j}(x)u_{j}(t),\mbox{ }%
u_{j}(t)\in \mathbb{R}\mbox{ a.e. }t\in (0,\infty ),\mbox{ }j=1,...,m,
\label{44-1} \\
\alpha &=&(\alpha _{1},...,\alpha _{m})\in \left( L^{2}(\Gamma )\right) ^{m},%
\mbox{ }\alpha _{j}\geq 0\mbox{ a.e. }x\in \Gamma .  \notag
\end{eqnarray}%
We assume in addition that 
\begin{equation}
\frac{D_{0}\alpha _{j}}{x}\in L^{2}(\Omega ),\mbox{ }j=1,...,m.  \label{46}
\end{equation}%
The expression (\ref{44-1}) allows the possibility to consider combinations
of conditions on subsets of the boundary for the controls $u_{j}(t)\in 
\mathbb{R}.$ The hypothesis (\ref{46}) will be justified later.

$(i_{1})$ For this problem we choose%
\begin{equation}
H=W=Z=L^{2}(\Omega ),\mbox{ }U=\mathbb{R}^{m},  \label{46-1}
\end{equation}%
\begin{equation}
B_{1}w=\chi _{\omega _{1}}(x)w,\mbox{ }C_{1}y=\chi _{\Omega _{C}}(x)y,\mbox{
\ }D_{1}u=\sum\limits_{j=1}^{m}d_{j}(x)u_{j},\mbox{ }x\in \Omega ,
\label{44-2}
\end{equation}%
$u=(u_{1},...,u_{m}),$ with the conditions $\omega _{1}\sqsubseteq \Omega ,$ 
$\Omega _{0}\sqsubseteq \Omega _{C},$ and 
\begin{equation}
d_{j}\in L^{2}(\Omega ),\mbox{ }d_{j}(x)=0\mbox{ on }\Omega _{C},\mbox{ }%
\int_{\Omega \backslash \Omega _{C}}d_{j}d_{k}dx=\delta _{jk}.  \label{44-4}
\end{equation}%
Thus, $B_{1}\in L(L^{2}(\Omega ),L^{2}(\Omega )),$ $C_{1}\in L(L^{2}(\Omega
),L^{2}(\Omega ))$ and $D_{1}:U\rightarrow L^{2}(\Omega ).$ The operator $%
B_{2}$ will be further defined. The operator $A$ is the same as before, that
is%
\begin{equation}
A:D(A)\subset L^{2}(\Omega )\rightarrow L^{2}(\Omega ),\mbox{ }Ay=\Delta y+%
\frac{\lambda y}{|x|^{2}}+a(x)y,  \label{45-1}
\end{equation}%
\begin{equation}
D(A)=\left\{ y\in H_{0}^{1}(\Omega );\mbox{ }Ay\in L^{2}(\Omega )\right\} .
\label{45-2}
\end{equation}%
By Lemma 4.1, for $\lambda <H_{N}$ and $\omega >a_{0},$ it follows that $-A$
is $\omega $-$m$-accretive on $L^{2}(\Omega )$ and self-adjoint, so that $A$
generates a $C_{0}$ compact semigroup $e^{At}$ on $L^{2}(\Omega )$.
Moreover, as we shall see later, if $y\in D(A)$ then $y\in H^{2}(\Omega
\backslash \{0\}).$

In order to write equation\ (\ref{41}) in the operatorial form, we need some
preliminaries. Let us consider the problem 
\begin{equation}
\Delta \theta =0\mbox{ in }\Omega ,\mbox{ }\theta =v\mbox{ on }\Gamma ,\mbox{
for }t>0.  \label{47}
\end{equation}%
The boundary condition is meant in the sense of the trace of $\theta $ on $%
\Gamma ,$ generally denoted by $tr(\theta ).$ But, if any confusion is
avoided we shall no longer indicate the trace by the symbol $\mathit{tr}$.
The unique solution to this problem is the well-known Dirichlet map, $%
v\rightarrow \theta ,$ here denoted by $D_{0}v.$ If $v\in L^{2}(\Gamma ),$
then $D_{0}:L^{2}(\Gamma )\rightarrow H^{1/2}(\Omega )$ and it satisfies $%
\left\Vert D_{0}v\right\Vert _{H^{1/2}(\Omega )}\leq C\left\Vert
v\right\Vert _{L^{2}(\Gamma )}$ (see e.g. \cite{L-T}).

In our case, $v=\widetilde{u}\in L^{2}(\mathbb{R}_{+};L^{2}(\Gamma ))$ and
so $D_{0}\widetilde{u}(t)\in H^{1/2}(\Omega )$ and 
\begin{equation*}
\left\Vert D_{0}\widetilde{u}(t)\right\Vert _{H^{1/2}(\Omega )}\leq
C\left\Vert \widetilde{u}(t)\right\Vert _{L^{2}(\Gamma )},\mbox{ a.e. }t>0.
\end{equation*}%
Moreover, since $\widetilde{u}$ is given by (\ref{44-1}) and $D_{0}$ is
linear we have 
\begin{equation}
D_{0}\widetilde{u}(t)=\sum\limits_{j=1}^{m}u_{j}(t)D_{0}\alpha _{j},\mbox{ }%
t>0.  \label{47-1}
\end{equation}%
Let us introduce the operator 
\begin{equation}
A_{0}:D(A_{0})=D(A)\subset L^{2}(\Omega )\rightarrow L^{2}(\Omega ),\mbox{ }%
A_{0}y=\Delta y+\frac{\lambda y}{\left\vert x\right\vert ^{2}}.  \label{47-2}
\end{equation}%
This operator is $m$-dissipative on $L^{2}(\Omega )$ by a similar proof as
in Lemma 4.1. Let us determine the Dirichlet mapping $v\rightarrow Dv$
corresponding to $A_{0},$ that is 
\begin{equation}
\Delta Dv+\frac{\lambda Dv}{|x|^{2}}=0\mbox{ in }\Omega ,\mbox{ }Dv=v\mbox{
on }\Gamma .  \label{48}
\end{equation}

\begin{lemma}
For $\lambda <H_{N},$ $Dv$ associated to $A_{0}$ exists and it is unique for 
$v\in L^{2}(\Gamma )$ satisfying $\frac{D_{0}v}{x}\in L^{2}(\Omega ).$
Moreover, one has 
\begin{equation}
Dv\in H^{1/2}(\Omega )\mbox{ and }\left\Vert Dv\right\Vert _{H^{1/2}(\Omega
)}\leq C\left( \left\Vert v\right\Vert _{L^{2}(\Gamma )}+\left\Vert \frac{%
D_{0}v}{x}\right\Vert _{L^{2}(\Omega )}\right) .  \label{49}
\end{equation}
\end{lemma}

\begin{proof}
Let $t$ be fixed and denote $\varphi =Dv-D_{0}v$ and consider the equation 
\begin{equation}
\Delta \varphi +\frac{\lambda \varphi }{\left\vert x\right\vert ^{2}}=-\frac{%
\lambda D_{0}v}{\left\vert x\right\vert ^{2}}\mbox{ in }\Omega ,\mbox{ }%
\varphi =0\mbox{ on }\Gamma .  \label{50}
\end{equation}%
We assert that problem (\ref{50}) has a unique solution in $D(A)$ and prove
it via a variational technique, by showing that the solution to (\ref{50})
is given by the minimization of the functional $\Psi (\varphi ),$ 
\begin{equation}
\min_{\varphi \in H_{0}^{1}(\Omega )}\left\{ \Psi (\varphi )=\int_{\Omega
}\left( \frac{1}{2}\left\vert \nabla \varphi \right\vert ^{2}-\frac{1}{2}%
\frac{\lambda \varphi ^{2}}{\left\vert x\right\vert ^{2}}-\frac{\lambda
\varphi D_{0}v}{\left\vert x\right\vert ^{2}}\right) dx\right\} .  \label{51}
\end{equation}%
It is easily seen that 
\begin{equation*}
\left( \frac{1}{2}-\frac{\lambda }{H_{N}}\right) \int_{\Omega }\left\vert
\nabla \varphi \right\vert ^{2}dx-\lambda \int_{\Omega }\left\vert \frac{%
D_{0}v}{x}\right\vert ^{2}dx\leq \Psi (\varphi )<\infty ,
\end{equation*}%
so that $\Psi $ has an infimum $d.$ We note here the necessity of the
assumption $\frac{D_{0}v}{x}\in L^{2}(\Omega )$. Next, we proceed as in
Lemma 4.1 and show that $\varphi \in H_{0}^{1}(\Omega )$ is the unique weak
solution to the equation (\ref{50}). By (\ref{50}) we note that by
multiplying by $\varphi $ we get 
\begin{equation*}
\left\Vert \nabla \varphi \right\Vert _{2}^{2}+\frac{\lambda }{2}\left\Vert 
\frac{\varphi }{x}\right\Vert _{2}^{2}\leq \frac{\lambda }{2}\left\Vert 
\frac{D_{0}v}{x}\right\Vert _{2}^{2}.
\end{equation*}

Then, it follows that $Dv=\varphi +D_{0}v$ which is the Dirichlet map for (%
\ref{48}), has the properties $Dv\in H^{1/2}(\Omega ),$ $\frac{Dv}{x}=\frac{%
\varphi }{x}+\frac{D_{0}v}{x}\in L^{2}(\Omega )$ and $\left\Vert
Dv\right\Vert _{H^{1/2}(\Omega )}\leq \left\Vert \varphi +D_{0}v\right\Vert
_{H^{1/2}(\Omega )}\leq C\left( \left\Vert \varphi \right\Vert
_{H_{0}^{1}(\Omega )}+\left\Vert D_{0}v\right\Vert _{H^{1/2}(\Omega
)}\right) ,$ implying (\ref{49}).
\end{proof}

\noindent Lemma 5.1 implies that the operator $D:L^{2}(\Gamma )\rightarrow
L^{2}(\Omega )$ with the domain $\left\{ v\in L^{2}(\Gamma );\frac{D_{0}v}{x}%
\in L^{2}(\Omega )\right\} $ is closed and densely defined. We denote by $%
D^{\ast }:L^{2}(\Omega )\rightarrow L^{2}(\Gamma )$ its adjoint.

Now, we can write the operatorial form of the system. Let $%
u=(u_{1},...,u_{m})$ and assume for the beginning that 
\begin{equation*}
u\in W^{1,2}(0,T;\mathbb{R}^{m}),\mbox{ }w\in W^{1,2}(0,T;L^{2}(\Omega )),%
\mbox{ }T\geq 0
\end{equation*}%
and note that $D\widetilde{u}(t)$ is well defined due to (\ref{46}), $D%
\widetilde{u}(t)\in H^{1/2}(\Omega )$ and 
\begin{equation*}
\left\Vert D\widetilde{u}(t)\right\Vert _{H^{1/2}(\Omega )}\leq
C\sum\limits_{j=1}^{m}|u_{j}(t)|\left( \left\Vert \alpha _{j}\right\Vert
_{L^{2}(\Gamma )}+\left\Vert \frac{D_{0}\alpha _{j}}{x}\right\Vert
_{L^{2}(\Omega )}\right) ,\mbox{ a.e. }t>0.
\end{equation*}%
We and write the difference system (\ref{41}) and (\ref{48}),%
\begin{eqnarray*}
&&(y-D\widetilde{u})_{t}-\Delta (y-D\widetilde{u})-\frac{\lambda (y-D%
\widetilde{u})}{|x|^{2}}-a(x)(y-D\widetilde{u}) \\
&=&B_{1}w-(D\widetilde{u})_{t}+a(x)D\widetilde{u},\mbox{ in }(0,\infty
)\times \Omega , \\
y-D\widetilde{u} &=&0\mbox{, on }(0,\infty )\times \Gamma ,\mbox{ }(y-D%
\widetilde{u})(0)=y_{0}-\widetilde{\theta _{0}}\mbox{ in }\Omega ,
\end{eqnarray*}%
where $\widetilde{\theta _{0}}=D\widetilde{u}(0).$ The solution to the
previous system reads%
\begin{equation*}
(y-D\widetilde{u})(t)=e^{At}(y_{0}-\widetilde{\theta _{0}}%
)+\int_{0}^{t}e^{A(t-s)}(B_{1}w+aD\widetilde{u})(s)ds-%
\int_{0}^{t}e^{A(t-s)}(D\widetilde{u})_{t}(s)ds.
\end{equation*}%
Integrating by parts the last right-hand side term we obtain%
\begin{eqnarray*}
y(t)-D\widetilde{u}(t) &=&e^{At}y_{0}-e^{At}\widetilde{\theta _{0}}%
+\int_{0}^{t}e^{A(t-s)}(B_{1}w+a(x)D\widetilde{u})(s)ds \\
&&-D\widetilde{u}(t)+e^{At}\widetilde{\theta _{0}}-\int_{0}^{t}e^{A(t-s)}AD%
\widetilde{u}(s)dx
\end{eqnarray*}%
which yields 
\begin{equation*}
y(t)=e^{At}y_{0}-\int_{0}^{t}e^{A(t-s)}AD\widetilde{u}(s)dx+%
\int_{0}^{t}e^{A(t-s)}(B_{1}w+a(x)D\widetilde{u})(s)ds.
\end{equation*}%
The formula is preserved by density if $u\in L^{2}(0,T;\mathbb{R}^{m})$ and $%
w\in L^{2}(0,T;L^{2}(\Omega ))$ and this represents the solution to the
equation 
\begin{equation}
y^{\prime }(t)=Ay(t)+B_{1}w(t)-AD\widetilde{u}(t)+a(x)D\widetilde{u}(t),%
\mbox{ }y(0)=y_{0}.  \label{53}
\end{equation}%
Since $D\widetilde{u}(t)$ is not in $D(A)$ one must interpret $AD\widetilde{u%
}(t)$ by using the extension $\widetilde{A}$ of $A$ to the whole space $%
L^{2}(\Omega )$ by 
\begin{equation}
\widetilde{A}:L^{2}(\Omega )\rightarrow (D(A))^{\prime },\mbox{ }%
\left\langle \widetilde{A}y,\psi \right\rangle _{(D(A))^{\prime
},D(A)}=(y,A\psi ),\mbox{ }\forall \psi \in D(A),  \label{54}
\end{equation}%
see (\ref{401}). Now, we can define $B_{2}:U\rightarrow (D(A))^{\prime },$ 
\begin{equation}
B_{2}u=-\widetilde{A}\left( \sum\limits_{j=1}^{m}u_{j}D\alpha _{j}\right)
+a(x)\sum\limits_{j=1}^{m}u_{j}D\alpha
_{j}=-\sum\limits_{j=1}^{m}u_{j}A_{0}D\alpha _{j}\mbox{,}  \label{55}
\end{equation}%
where $u=(u_{1},...,u_{m})\in U=\mathbb{R}^{m}.$ Expression (\ref{55}) is
well defined since $D\alpha _{j}\in H^{1/2}(\Omega )\subset L^{2}(\Omega )$
and $a\in L^{\infty }(\Omega ).$ Eventually, we can express equations (\ref%
{41})-(\ref{42}) as%
\begin{eqnarray}
y^{\prime }(t) &=&Ay(t)+B_{1}w(t)+B_{2}\widetilde{u}(t),\mbox{ }t\geq 0,%
\mbox{ }  \label{56} \\
y(0) &=&y_{0}  \notag
\end{eqnarray}%
with $\widetilde{A}$ defined in (\ref{54}), $B_{2}$ defined in (\ref{55})
and $\widetilde{u}$ defined in (\ref{44-1}).

\medskip

$(i_{2})$ For verifying (\ref{12}) we need to calculate $B_{2}^{\ast }$ and $%
D^{\ast }.$ We denote by $\frac{\partial v}{\partial \nu }$ the normal
derivative of $v$ on the boundary $\Gamma .$ We give the following lemma.

\begin{lemma}
The operator $B_{2}^{\ast }:D(A)\rightarrow \mathbb{R}^{m}$ is given by 
\begin{equation}
(B_{2}^{\ast }v)_{j}=-\left( \alpha _{j},\frac{\partial v}{\partial \nu }%
\right) _{L^{2}(\Gamma )}\mbox{, for }v\in D(A),\mbox{ }j=1,...,m,
\label{60}
\end{equation}%
where $\frac{\partial v}{\partial \nu }\in L^{2}(\Gamma ).$

The operator $D^{\ast }:L^{2}(\Omega )\rightarrow L^{2}(\Gamma ),$ is
defined by%
\begin{equation}
D^{\ast }p=\frac{\partial }{\partial \nu }(A_{0}^{-1}p)\mbox{ on }\Gamma ,%
\mbox{ for }p\in L^{2}(\Omega ).  \label{59-1}
\end{equation}
\end{lemma}

\begin{proof}
We use the definition of $B_{2}$ and for $v\in D(A)$ we calculate 
\begin{eqnarray}
&&\left\langle B_{2}u,v\right\rangle _{(D(A))^{\prime },D(A)}=\left\langle -%
\widetilde{A}\left( \sum\limits_{j=1}^{m}u_{j}D\alpha _{j}\right)
+a\sum\limits_{j=1}^{m}u_{j}D\alpha _{j},v\right\rangle _{(D(A))^{\prime
},D(A)}  \label{57} \\
&=&-\sum\limits_{j=1}^{m}\left( u_{j}D\alpha _{j},Av\right) _{L^{2}(\Omega
)}+\sum\limits_{j=1}^{m}\left( u_{j}D\alpha _{j},av\right) _{L^{2}(\Omega
)}=\sum\limits_{j=1}^{m}u_{j}\left( D\alpha _{j},-Av+av\right)
_{L^{2}(\Omega )}  \notag \\
&=&u\cdot (D\alpha ,-Av+av)_{L^{2}(\Omega )}=u\cdot (D\alpha
,-A_{0}v)_{L^{2}(\Omega )},  \notag
\end{eqnarray}%
where $(D\alpha ,A_{0}v)_{L^{2}(\Omega )}$ denotes the vector with the
components $(D\alpha _{j},-A_{0}v)_{L^{2}(\Omega )}$ for $v\in D(A).$ Here
we took into account that $-Av+av=-A_{0}v$ with $A_{0}$ defined in (\ref%
{47-2}). Hence, we can define the components of $B_{2}^{\ast
}:D(A)\rightarrow U^{\ast }=U=\mathbb{R}^{m}$ by%
\begin{equation}
(B_{2}^{\ast }v)_{j}=\left( D\alpha _{j},-A_{0}v\right) _{L^{2}(\Omega )},%
\mbox{ }v\in D(A),\mbox{ }j=1,...,m.  \label{57-1}
\end{equation}

For the computation of $(D\alpha _{j},-A_{0}v)_{L^{2}(\Omega )}$ let us
consider the generic systems%
\begin{equation}
\Delta D\beta +\frac{\lambda D\beta }{\left\vert x\right\vert ^{2}}=0,\mbox{ 
}D\beta =\beta \mbox{ on }\Gamma ,\mbox{ }\beta \in L^{2}\left( \Gamma
\right) ,  \label{58}
\end{equation}%
\begin{equation}
-\Delta v-\frac{\lambda v}{\left\vert x\right\vert ^{2}}=p,\mbox{ }v=0\mbox{
on }\Gamma ,\mbox{ }p\in L^{2}(\Omega ).  \label{58-1}
\end{equation}%
The second system has a unique solution $v\in H_{0}^{1}(\Omega ).$ In order
to make a rigorous calculus we assume first that $\beta \in H^{1}\left(
\Gamma \right) $ and $-A_{0}$ is replaced by for $\varepsilon >0$ by 
\begin{equation}
-A_{0,\varepsilon }=-\Delta -\frac{\lambda }{\left\vert x\right\vert
^{2}+\varepsilon },\mbox{ }D(A_{0,\varepsilon })=H^{2}(\Omega )\cap
H_{0}^{1}(\Omega ).  \label{58-2}
\end{equation}%
Thus, the equation $-A_{0,\varepsilon }v=p$ has a unique solution $%
v_{\varepsilon }\in H^{2}(\Omega )\cap H_{0}^{1}(\Omega )$ and all
operations below make sense. We multiply the approximating equation for $%
D\beta $ by the solution $v_{\varepsilon }.$ By applying the Green's formula
we obtain%
\begin{equation*}
\int_{\Omega }\left( D\beta \Delta v_{\varepsilon }+\frac{\lambda
v_{\varepsilon }D\alpha _{j}}{\left\vert x\right\vert ^{2}+\varepsilon }%
\right) dx+\int_{\Gamma }\left( v_{\varepsilon }\frac{\partial D\beta }{%
\partial \nu }-D\beta \frac{\partial v_{\varepsilon }}{\partial \nu }\right)
dx=0
\end{equation*}%
which implies, by using (\ref{58-1}) and the boundary condition for $D\alpha
_{j},$ that 
\begin{equation*}
-\int_{\Omega }pD\beta dx=\int_{\Gamma }\beta \frac{\partial v_{\varepsilon }%
}{\partial \nu }d\sigma ,\mbox{ }\forall \beta \in L^{2}(\Gamma ).
\end{equation*}%
Therefore, we have for each $p\in L^{2}(\Omega )$ 
\begin{equation*}
(D\beta ,p)_{L^{2}(\Omega )}=\left( \beta ,\frac{\partial }{\partial \nu }%
(A_{0,\varepsilon }^{-1}p)\right) _{L^{2}(\Gamma )}\mbox{ }\forall \beta \in
H^{1}\left( \Gamma \right) ,
\end{equation*}%
which can be written also as%
\begin{equation*}
(D\beta ,-A_{0,\varepsilon }v_{\varepsilon })_{L^{2}(\Omega )}=-\left( \beta
,\frac{\partial v_{\varepsilon }}{\partial \nu }\right) _{L^{2}(\Gamma )}%
\mbox{ for }\beta \in H^{1}\left( \Gamma \right) ,\mbox{ }v_{\varepsilon
}\in D(A_{0,\varepsilon }).
\end{equation*}%
These remain true at limit as $\varepsilon \rightarrow 0,$ hence%
\begin{equation}
(D\beta ,p)_{L^{2}(\Omega )}=\left( \beta ,\frac{\partial }{\partial \nu }%
(A_{0}^{-1}p)\right) _{L^{2}(\Gamma )}\mbox{ for }\beta \in H^{1}\left(
\Gamma \right) ,\mbox{ }p\in L^{2}(\Omega ),  \label{58-0}
\end{equation}%
\begin{equation}
(D\beta ,-A_{0}v)_{L^{2}(\Omega )}=-\left( \beta ,\frac{\partial v}{\partial
\nu }\right) _{L^{2}(\Gamma )}\mbox{ for }\beta \in H^{1}\left( \Gamma
\right) ,\mbox{ }v\in D(A_{0})=D(A)  \label{59}
\end{equation}%
and the latter makes sense since $\frac{\partial v}{\partial \nu }\in
H^{-1/2}(\Gamma ).$ We note that both $A_{0,\varepsilon }$ and $A_{0}$ are
surjective, because they are $m$-accretive and coercive. Then, by (\ref{58-0}%
) we can define $D^{\ast }:L^{2}(\Omega )\rightarrow L^{2}(\Gamma ),$ by (%
\ref{59-1}).

Going back to (\ref{57-1}) and using (\ref{59}) in which we set $\beta
:=\alpha _{j}$ it turns out that we can define $B_{2}^{\ast
}:D(A)\rightarrow U$ by 
\begin{equation}
(B_{2}^{\ast }v)_{j}=(D\alpha _{j},-A_{0}v)_{L^{2}(\Gamma )}=-\left( \alpha
_{j},\frac{\partial v}{\partial \nu }\right) _{L^{2}(\Gamma )}\mbox{, for }%
v\in D(A).  \label{59-0}
\end{equation}%
It remains to show that $\frac{\partial v}{\partial \nu }$ belongs to $%
L^{2}(\Gamma )$ if $v\in D(A).$ Indeed, there exists $(v_{\varepsilon
})_{\varepsilon }\subset H^{2}(\Omega )\cap D(A)$ such that $v_{\varepsilon
}\rightarrow v$ strongly in $D(A)$, $\frac{\partial v_{\varepsilon }}{%
\partial \nu }\rightarrow \frac{\partial v}{\partial \nu }$ strongly in $%
H^{-1/2}(\Gamma )$ as $\varepsilon \rightarrow 0$ and 
\begin{equation}
(B_{2}^{\ast }v_{\varepsilon })_{j}=-\left( \alpha _{j},\frac{\partial
v_{\varepsilon }}{\partial \nu }\right) _{L^{2}(\Gamma )}.  \label{60-00}
\end{equation}%
We recall that $0\in \Omega $. We consider $\varphi \in C^{4}(\overline{%
\Omega })$ defined by 
\begin{equation*}
\varphi (x)=\left\{ 
\begin{array}{l}
0,\mbox{ if }x\in \Omega _{\delta } \\ 
1,\mbox{ if }x\in \Omega \backslash \Omega _{2\delta }%
\end{array}%
\right.
\end{equation*}%
where $\delta >0$ is such that $\Omega _{\delta }=\{x\in \Omega ;$ $%
\left\Vert x\right\Vert <\delta \}$ and $0\in \Omega _{\delta }.$ The
function $\varphi v_{\varepsilon }\in H^{2}(\Omega \backslash \Omega
_{2\delta }).$ Indeed, since $v_{\varepsilon }\in H^{2}(\Omega )$ it follows
that there exists $f\in L^{2}(\Omega )$ such that $f=Av_{\varepsilon }$ and
so $\Delta v_{\varepsilon }=f-\frac{\lambda v_{\varepsilon }}{\left\vert
x\right\vert ^{2}}\in L^{2}(\Omega \backslash \Omega _{2\delta }).$ We have%
\begin{equation*}
\Delta (\varphi v_{\varepsilon })=\varphi \Delta v_{\varepsilon }+2\nabla
\varphi \cdot \nabla v_{\varepsilon }+v_{\varepsilon }\Delta \varphi \in
L^{2}(\Omega ).
\end{equation*}%
This together with the boundary condition $\varphi v_{\varepsilon }=0$ on $%
\Gamma $ implies that $\varphi v_{\varepsilon }\in H^{2}(\Omega \backslash
\Omega _{2\delta })$ and so $v_{\varepsilon }\in H^{2}(\Omega \backslash
\Omega _{2\delta }),$ too, because $\varphi =1$ on $\Omega \backslash \Omega
_{2\delta }.$ Consequently, $\frac{\partial v_{\varepsilon }}{\partial \nu }%
\in H^{1/2}(\Gamma )\subset L^{2}(\Gamma ).$ This is preserved by density
nearby the boundary. Finally, (\ref{59-0}) remains true by density for $%
\alpha _{j}\in L^{2}(\Gamma )$ and so this implies (\ref{60}).
\end{proof}

\medskip

Now, we pass to the proof of $(i_{2}).$ Such a result is proved for the
Laplace operator in \cite{vbp}, p. 320, Proposition 4.39, but here we give a
complete different proof under our hypotheses.

To this end, we recall that $Ay=A_{0}y+ay$ with $A_{0}$ defined in (\ref%
{47-2}) and consider the problem 
\begin{equation}
\frac{dy}{dt}(t)+B_{0}y(t)-ay=0,\mbox{ in }(0,T)\times \Omega ,\mbox{ }%
y(0)=y_{0}\in L^{2}(\Omega )  \label{60-0}
\end{equation}%
where 
\begin{equation}
B_{0}=-A_{0},\mbox{ }B_{0}=-\Delta -\frac{\lambda }{\left\vert x\right\vert
^{2}},\mbox{ }B_{0}:D(B_{0})=D(A_{0})\rightarrow L^{2}(\Omega ).
\label{60-B0}
\end{equation}%
The operator $B_{0}$ is $m$-accretive, $B_{0}=B_{0}^{\ast }$ and $B_{0}-aI$
is $\omega $-$m$-accretive. The unique solution to problem (\ref{60-0}) has
also the property $y(t)\in D(A)=D(A_{0})$ a.e. $t\in (0,T)$ by the
regularizing effect (see \cite{VB-book-2010}, p. 158 Theorem 4.11).

First, we determine two estimates. We multiply equation (\ref{60-0}) first
by $y(t)$ and integrate over $(0,t).$ We obtain, using Gronwall's lemma 
\begin{equation}
\left\Vert y(t)\right\Vert
_{2}^{2}+\int_{0}^{t}(B_{0}y(s),y(s))_{2}ds=C_{T}\left\Vert y_{0}\right\Vert
_{2}^{2},\mbox{ }\forall t\in \lbrack 0,T].  \label{60-1}
\end{equation}%
Then, we multiply (\ref{60-0}) by $tB_{0}y(t)$ which yields%
\begin{equation}
\frac{1}{2}\frac{d}{dt}\left( tB_{0}y(t),y(t)\right) _{2}+t\left\Vert
B_{0}y(t)\right\Vert _{2}^{2}=\frac{1}{2}\left( B_{0}y(t),y(t)\right)
_{2}+(ay(t),B_{0}y(t))_{2}.  \label{60-2}
\end{equation}%
We integrate this and by (\ref{60-1}) we get 
\begin{equation}
t(B_{0}y(t),y(t))_{2}+\int_{0}^{t}s\left\Vert B_{0}y(s)\right\Vert
_{2}^{2}ds\leq C\int_{0}^{t}(B_{0}y(s),y(s))_{2}ds\leq C_{T}\left\Vert
y_{0}\right\Vert _{2}^{2}.  \label{60-3}
\end{equation}%
To prove $(i_{2})$ we have to estimate 
\begin{eqnarray}
\left\Vert B_{2}^{\ast }e^{At}y_{0}\right\Vert _{\mathbb{R}^{m}}
&=&\left\Vert B_{2}^{\ast }y(t)\right\Vert _{\mathbb{R}^{m}}=\left\Vert
\left( -\left( \alpha _{j},\frac{\partial y(t)}{\partial \nu }\right)
_{L^{2}(\Gamma )}\right) _{j=1}^{m}\right\Vert _{\mathbb{R}^{m}}
\label{60-01} \\
&\leq &\sum\limits_{j=1}^{m}\left\Vert \alpha _{j}\right\Vert _{L^{2}(\Gamma
)}\left\Vert \frac{\partial y(t)}{\partial \nu }\right\Vert _{L^{2}(\Gamma
)},  \notag
\end{eqnarray}%
thus, actually we have to estimate $\left\Vert \frac{\partial y(t)}{\partial
\nu }\right\Vert _{L^{2}(\Gamma )}$ for $t>0.$ Since we shall relate this to
the fractional powers of the operator $B_{0},$ for a rigorous computation
involving its fractional powers we shall rely again on the approximation, $%
B_{0,\varepsilon }=-A_{0,\varepsilon },$ see (\ref{58-2}). We proceed with
all calculations for the approximating equation (\ref{60-0}) with $%
B_{0,\varepsilon }$ instead of $B_{0}$ and pass to the limit at the end.
Thus, $D(B_{0,\varepsilon })=H^{2}(\Omega )\cap H_{0}^{1}(\Omega ),$ $%
B_{0,\varepsilon }:D(B_{0,\varepsilon })\subset L^{2}(\Omega )\rightarrow
L^{2}(\Omega )$ and it is $m$-accretive and self-adjoint.

Therefore, we recall that the fractional powers are defined by $%
B_{0,\varepsilon }^{s}:D(B_{0,\varepsilon }^{s})\subset L^{2}(\Omega
)\rightarrow L^{2}(\Omega )$, $s\geq 0,$ see \cite{Pazy}$.$ Then, $%
D(B_{0,\varepsilon }^{s})\subset H^{2s}(\Omega )$ with equality iff $2s<3/2,$
see e.g., \cite{Fujiwara}. We have the interpolation inequality%
\begin{equation}
\left\Vert B_{0,\varepsilon }^{s}w\right\Vert _{2}\leq C\left\Vert
B_{0,\varepsilon }^{s_{1}}w\right\Vert _{2}^{\lambda }\left\Vert
B_{0,\varepsilon }^{s_{2}}w\right\Vert _{2}^{1-\lambda },\mbox{ for }%
s=\lambda s_{1}+(1-\lambda )s_{2},  \label{300}
\end{equation}%
and the relations%
\begin{equation}
\left\Vert B_{0,\varepsilon }^{s}w\right\Vert _{2}\leq C\left\Vert
B_{0,\varepsilon }^{s_{1}}w\right\Vert _{2}\mbox{ if }s<s_{1},  \label{301}
\end{equation}%
\begin{equation}
\left\Vert B_{0,\varepsilon }^{s}w\right\Vert _{H^{m}(\Omega )}\leq
C\left\Vert B_{0,\varepsilon }^{s+m/2}w\right\Vert _{2}.  \label{302}
\end{equation}

Now, we come back to $\frac{\partial y}{\partial \nu }(t)$ and using the
trace theorem and (\ref{302}) applied to $B_{0,\varepsilon }$ we write for
the approximating solution 
\begin{equation}
\left\Vert \frac{\partial y_{\varepsilon }}{\partial \nu }(t)\right\Vert
_{L^{2}(\Gamma )}\leq C\left\Vert y_{\varepsilon }(t)\right\Vert
_{H^{3/2}(\Omega )}\leq C\left\Vert B_{0,\varepsilon }^{3/4}y(t)\right\Vert
_{L^{2}(\Omega )},  \label{60-02}
\end{equation}%
so that we must estimate $\left\Vert B_{0,\varepsilon }^{3/4}y(t)\right\Vert
_{H}.$

Next, we use (\ref{300}) and write 
\begin{equation}
\left\Vert B_{0,\varepsilon }^{3/4}y_{\varepsilon }(t)\right\Vert _{2}\leq
C\left\Vert B_{0,\varepsilon }y_{\varepsilon }(t)\right\Vert
_{2}^{3/4}\left\Vert y_{\varepsilon }(t)\right\Vert _{2}^{1/4}.  \label{60-4}
\end{equation}%
Further, we calculate via H\"{o}lder's inequality%
\begin{eqnarray}
&&\int_{0}^{t}\left\Vert B_{0,\varepsilon }y_{\varepsilon }(s)\right\Vert
_{2}^{3/4}ds=\int_{0}^{t}s^{p}\left\Vert B_{0,\varepsilon }y_{\varepsilon
}(s)\right\Vert _{2}^{3/4}s^{-p}ds  \label{60-5} \\
&\leq &\left( \int_{0}^{t}s^{8p/3}\left\Vert B_{0,\varepsilon
}y_{\varepsilon }(t)\right\Vert _{2}^{2}ds\right) ^{3/8}\left(
\int_{0}^{t}s^{-8p/5}ds\right) ^{5/8}  \notag \\
&=&\left( \int_{0}^{t}s\left\Vert B_{0,\varepsilon }y_{\varepsilon
}(s)\right\Vert _{2}^{2}ds\right) ^{3/8}\left( \int_{0}^{t}s^{-3/5}ds\right)
^{5/8}  \notag \\
&\leq &C\left( \int_{0}^{t}s\left\Vert B_{0,\varepsilon }y_{\varepsilon
}(s)\right\Vert _{2}^{2}ds\right) ^{3/8}\left( t^{2/5}\right) ^{5/8},  \notag
\end{eqnarray}%
where we chose $p=\frac{3}{8}.$ This together with (\ref{60-01}), (\ref%
{60-02}), (\ref{60-4}) and (\ref{60-1}) implies%
\begin{eqnarray}
&&\int_{0}^{t}\left\Vert B_{2,\varepsilon }^{\ast }e^{As}y_{0}\right\Vert _{%
\mathbb{R}^{m}}ds\leq C\int_{0}^{t}\left\Vert \frac{\partial y_{\varepsilon }%
}{\partial \nu }(s)\right\Vert _{L^{2}(\Gamma )}ds\leq
C\int_{0}^{t}\left\Vert B_{0,\varepsilon }y_{\varepsilon }(s)\right\Vert
_{L^{2}(\Omega )}^{3/4}ds  \label{60-7} \\
&\leq &C\int_{0}^{t}\left\Vert B_{0,\varepsilon }y_{\varepsilon
}(s)\right\Vert _{2}^{3/4}\left\Vert y_{\varepsilon }(s)\right\Vert
_{2}^{1/4}ds\leq C_{T}\left\Vert y_{0}\right\Vert
_{2}^{1/4}\int_{0}^{t}\left\Vert B_{0,\varepsilon }y_{\varepsilon
}(s)\right\Vert _{2}^{3/4}ds  \notag \\
&\leq &C_{T}\left\Vert y_{0}\right\Vert _{2}^{1/4}\left\Vert
y_{0}\right\Vert _{2}^{3/4}\left( t^{2/5}\right) ^{5/8}\leq C_{T}\left\Vert
y_{0}\right\Vert _{2},\mbox{ }\forall t\in \lbrack 0,T].  \notag
\end{eqnarray}%
Passing to the limit by recalling (\ref{59}) we get $(i_{2})$ as claimed.

This hypothesis has also an important consequence. We note that (\ref{56})
with the initial condition $y(0)=y_{0}\in L^{2}(\Omega )$ has a unique
solution $y\in C([0,T];(D(A))^{\prime }),$%
\begin{equation}
y(t)=e^{At}y_{0}+\int_{0}^{t}e^{A(t-s)}(B_{1}w(s)+B_{2}u(s))ds,\mbox{ }t\in
\lbrack 0,\infty ).  \label{60-6}
\end{equation}
We are going to show first that $(i_{2})$ ensures in addition that $y\in
L^{2}(0,T;L^{2}(\Omega )).$

Actually, we shall prove the following assertion: if (\ref{12}) takes place
then the solution $y$ to (\ref{56}) belongs to $L^{2}(0,T;L^{2}(\Omega ))$
if $u\in L^{2}(0,T;U).$ Since in (\ref{60-6}) the sum between the first and
the last term corresponding to the contribution of $w$ is already in $%
C([0,T];L^{2}(\Omega ))$ we focus only on the term $Y(t):=%
\int_{0}^{t}e^{A(t-s)}B_{2}u(s)ds$ and show as in (\ref{16-0}) that $%
\left\Vert Y\right\Vert _{L^{2}(0,T;L^{2}(\Omega ))}\leq C\left\Vert
u\right\Vert _{L^{2}(0,T;U)}.$ In conclusion, equation (\ref{56}) with the
initial condition $y_{0}\in L^{2}(\Omega )$ has a mild solution $y\in
L^{2}(0,T;L^{2}(\Omega ))$.

\medskip

$(i_{3})$ The first part of hypothesis $(i_{3}),$ that is the detectability
of the pair $(A,C_{1})$ follows as in Lemma 4.2. Now we prove (\ref{12-0}).
We recall that 
\begin{equation*}
A_{1}y=A_{0}y+a_{0}\chi _{\Omega _{0}}(x)y-k\chi _{\Omega _{C}}(x)y
\end{equation*}%
with $A_{0}$ defined in (\ref{47-2}) and consider the problem 
\begin{equation}
\frac{dy}{dt}(t)+B_{0}y(t)=a_{0}\chi _{\Omega _{0}}(x)y-k\chi _{\Omega
_{C}}(x)y,\mbox{ in }(0,T)\times \Omega ,\mbox{ }y(0)=y_{0}\in L^{2}(\Omega )
\label{61}
\end{equation}%
where $B_{0}=-A_{0}$ is $m$-accretive, $B_{0}=B_{0}^{\ast }$ and $A_{1}$ is $%
m$-accretive. Then, problem (\ref{61}) has a unique solution $%
y(t)=S_{1}(t)y_{0},$ where $S_{1}(t)$ is the $C_{0}$-semigroup generated by $%
A_{1}$. The solution $y\in L^{2}(0,T;H_{0}^{1}(\Omega ))$ and $y(t)\in D(A)$
a.e. $t\in (0,T)$.

Since $A_{1}=A+KC_{1}$ generates an exponentially stable semigroup we have 
\begin{equation}
\left\Vert y(t)\right\Vert _{2}\leq e^{-\alpha t}\left\Vert y_{0}\right\Vert
_{2},\mbox{ }\alpha =k-a_{0}.  \label{62}
\end{equation}%
Moreover, $S_{1}(t)$ is analytic and so 
\begin{equation}
\left\Vert A_{1}y(t)\right\Vert _{2}\leq \frac{C_{T}}{t}\left\Vert
y(t)\right\Vert _{2},\mbox{ }\forall t\in (0,T).  \label{63-0}
\end{equation}%
Since $\left\Vert B_{0}y\right\Vert _{H}\leq \left\Vert A_{1}y\right\Vert
_{H}+C\left\Vert y\right\Vert _{H}$ it follows that 
\begin{equation}
\left\Vert B_{0}y(t)\right\Vert _{2}\leq \frac{C_{T}}{t}\left\Vert
y(t)\right\Vert _{2},\mbox{ }\forall t\in (0,T).  \label{63}
\end{equation}%
The previous calculations for proving point $(i_{2})$ hold here too, and by (%
\ref{60-7}) we have 
\begin{eqnarray}
&&\left\Vert B_{2}^{\ast }e^{(A+KC_{1})t}y_{0}\right\Vert _{\mathbb{R}%
^{m}}=\left\Vert B_{2}^{\ast }y(t)\right\Vert _{\mathbb{R}^{m}}=\left\Vert
\left( -\left( \alpha _{j},\frac{\partial y(t)}{\partial \nu }\right)
_{L^{2}(\Gamma )}\right) _{j=1}^{m}\right\Vert _{\mathbb{R}^{m}}  \label{64}
\\
&\leq &\sum\limits_{j=1}^{m}\left\Vert \alpha _{j}\right\Vert _{L^{2}(\Gamma
)}\left\Vert \frac{\partial y(t)}{\partial \nu }\right\Vert _{L^{2}(\Gamma
)}\leq C\left\Vert B_{0}y(t)\right\Vert _{2}^{3/4}\left\Vert
y_{0}\right\Vert _{2}^{1/4}e^{-\alpha t/4},  \notag
\end{eqnarray}%
where $y(t)=S_{1}(t)y_{0}$ is the solution to (\ref{61}). Thus, 
\begin{equation}
\int_{0}^{T}\left\Vert B_{2}^{\ast }e^{(A+KC_{1})t}y_{0}\right\Vert _{%
\mathbb{R}^{m}}dt\leq C_{T}\left\Vert y_{0}\right\Vert _{2},\mbox{ for }%
T\geq 0.  \label{64-0}
\end{equation}%
On the other hand, for $t>T$ we have 
\begin{equation*}
\left\Vert A_{1}y(t)\right\Vert _{2}=\left\Vert
A_{1}S_{1}(T)S_{1}(t-T)y(t)\right\Vert _{2}\leq \frac{C_{T}}{T}\left\Vert
S_{1}(t-T)y(t)\right\Vert _{2}\leq \frac{C_{T}}{T}e^{-\alpha
(t-T)}\left\Vert y_{0}\right\Vert _{2}.
\end{equation*}%
Then we calculate 
\begin{eqnarray*}
&&\left\Vert B_{0}y(t)\right\Vert _{H}^{3/4}\leq \left( \left\Vert
A_{1}y(t)\right\Vert _{H}+C\left\Vert y(t)\right\Vert _{H}\right) ^{3/4}\leq
C\left\Vert A_{1}y(t)\right\Vert _{H}^{3/4}+C\left\Vert y(t)\right\Vert
_{H}^{3/4} \\
&\leq &\frac{C_{T}}{T^{3/4}}e^{-3\alpha (t-T)/4}\left\Vert y_{0}\right\Vert
_{2}^{3/4}+C\left\Vert y_{0}\right\Vert _{2}^{3/4},
\end{eqnarray*}%
hence, by (\ref{64})%
\begin{eqnarray}
\left\Vert B_{2}^{\ast }e^{(A^{\ast }+KC_{1})t}y_{0}\right\Vert _{U} &\leq
&\left( \frac{C_{T}}{T^{3/4}}e^{-3\alpha (t-T)/4}+1\right) \left\Vert
y_{0}\right\Vert _{2}^{3/4}\left\Vert y_{0}\right\Vert _{2}^{1/4}e^{-\alpha
t/4}  \label{64-2} \\
&=&\left( \frac{C_{T}}{T^{3/4}}e^{-3\alpha (t-T)/4}+e^{-\alpha t/4}\right)
\left\Vert y_{0}\right\Vert _{2},\mbox{ for }t>T.  \notag
\end{eqnarray}%
In particular, let $T=1$ and by (\ref{64-0}) and (\ref{64-2}) we finally get 
\begin{eqnarray}
&&\int_{0}^{\infty }\left\Vert B_{2}^{\ast }e^{(A^{\ast
}+KC_{1})t}y_{0}\right\Vert _{U}dt  \label{64-1} \\
&=&\int_{0}^{1}\left\Vert B_{2}^{\ast }e^{(A^{\ast
}+KC_{1})t}y_{0}\right\Vert _{U}dt+\int_{1}^{\infty }\left\Vert B_{2}^{\ast
}e^{(A^{\ast }+KC_{1})t}y_{0}\right\Vert _{U}dt  \notag \\
&\leq &C_{1}\left\Vert y_{0}\right\Vert _{2}+\left\Vert y_{0}\right\Vert
_{2}\int_{1}^{\infty }\left( C_{1}e^{-3\alpha (t-1)/4}+e^{-\alpha
t/4}\right) dt\leq C\left\Vert y_{0}\right\Vert _{2}\mbox{,}  \notag
\end{eqnarray}%
for all $y_{0}\in L^{2}(\Omega )$. In conclusion, we have obtained (\ref%
{12-0}) as claimed.

\medskip

$(i_{4})$ The adjoint of $D_{1}$ is $D_{1}^{\ast }:L^{2}(\Omega )\rightarrow 
\mathbb{R}^{m}$ 
\begin{equation*}
D_{1}^{\ast }v=\left( \int_{\Omega }d_{1}(x)v(x)dx,...,\int_{\Omega
}d_{m}(x)v(x)dx\right) .
\end{equation*}%
Then, by (\ref{44-4}), $\left\Vert D_{1}u\right\Vert _{L^{2}(\Omega
)}^{2}=\int_{\Omega }\left( \sum\limits_{j=1}^{m}d_{j}(x)\right) ^{2}dx=1,$
and $\int_{\Omega }d_{j}(x)\chi _{\Omega _{C}}(x)ydx=0,$ hence $D_{1}^{\ast
}C_{1}y(\xi )=0.$

\medskip

Then, calculating the operators in (\ref{15}) we see that formulae (\ref%
{36-0}), (\ref{36-2})-(\ref{36-3}) are the same and using (\ref{59-0}) we
get 
\begin{equation*}
PB_{2}B_{2}^{\ast }P\varphi (x)=\int_{\Omega }\varphi (\xi )\left(
\sum\limits_{j=1}^{m}A_{j}(\xi )A_{j}(x)\right) d\xi
\end{equation*}%
where 
\begin{equation*}
A_{j}(\xi )=\int_{\Gamma }\alpha _{j}(\sigma )\frac{\partial P_{0}}{\partial
\nu _{\sigma }}(\sigma ,\xi )d\sigma .
\end{equation*}%
Proceedings with all calculations as in Section \ref{Distributed} we have

\begin{theorem}
Let $\gamma >0$ and let $A,$ $B_{1},$ $C_{1}$ and $D_{1}$ be given by (\ref%
{45-1}) and (\ref{44-2}), respectively and $B_{2},$ $B_{2}^{\ast }$ be given
by (\ref{55}) and (\ref{60}). Assume that $P_{0}\in D(A)\times D(A)$ is a
solution to equation 
\begin{eqnarray}
&&\Delta _{x}P_{0}(x,\xi )+\Delta _{\xi }P_{0}(x,\xi )+\lambda P_{0}(x,\xi
)\left( \frac{1}{|x|^{2}}+\frac{1}{|\xi |^{2}}\right) +(a(x)+a(\xi
))P_{0}(x,\xi )  \notag \\
&&-\sum\limits_{j=1}^{m}A_{j}(x)A_{j}(\xi )+\gamma ^{-2}\int_{\Omega }\chi
_{\omega _{1}}(\overline{\xi })P_{0}(x,\overline{\xi })P_{0}(\overline{\xi }%
,\xi )d\overline{\xi }  \label{64-4} \\
&=&-\delta (x-\xi )\chi _{\Omega _{C}}(\xi ),\mbox{ in }\mathcal{D}^{\prime
}(\Omega \times \Omega ),  \notag
\end{eqnarray}%
with conditions (\ref{38})-(\ref{40}). Then, the feedback control $%
\widetilde{F}\in L(L^{2}(\Omega ),\mathbb{R}^{m}),$ 
\begin{equation}
(\widetilde{F}y)_{j}=\int_{\Omega }y(\xi )\left( \alpha _{j},\frac{\partial
P_{0}}{\partial \nu }(\cdot ,\xi )\right) _{L^{2}(\Gamma )}d\xi ,\mbox{ }%
j=1,...m,\mbox{ }\forall y\in L^{2}(\Omega )  \label{64-5}
\end{equation}%
solves the $H^{\infty }$-problem.
\end{theorem}

In this case, by (\ref{54}), (\ref{55}) and (\ref{57-1}) we have%
\begin{equation*}
\Lambda _{P}y=A_{0}\left( y+\int_{\Omega }y(\xi )\sum\limits_{j=1}^{m}\left(
\int_{\Gamma }\alpha _{j}(\sigma )\frac{\partial P_{0}}{\partial \nu
_{\sigma }}(\sigma ,\xi )d\sigma \right) D\alpha _{j}d\xi \right) +ay+\chi
_{\omega _{1}}\int_{\Omega }P_{0}(x,\xi )y(\xi )d\xi
\end{equation*}%
and we see that 
\begin{equation*}
D(\Lambda _{P})=\left\{ y\in H;\mbox{ }y+\int_{\Omega }y(\xi
)\sum\limits_{j=1}^{m}\left( \int_{\Gamma }\alpha _{j}(\sigma )\frac{%
\partial P_{0}}{\partial \nu _{\sigma }}(\sigma ,\xi )d\sigma \right)
D\alpha _{j}d\xi \in D(A)\right\} .
\end{equation*}%
Moreover, $\Lambda _{P}$ is closed because if $y_{n}\rightarrow y$ in $H$,
since $A_{0}$ is closed we see that $\Lambda _{P}y_{n}\rightarrow \Lambda
_{P}y$ in $H.$ Then, by Lemma \ref{operators} we deduce that $\mathcal{X}%
=D(\Lambda _{P}).$

\section{Dirichlet boundary control in an $1D$ domain with a boundary
singularity\label{1D}}

\setcounter{equation}{0}

We briefly discuss here the $H^{\infty }$-boundary control problem for an
one-dimensional parabolic equation with the singularity on the boundary.
Namely, let $\Omega =(0,1)$ and consider the system%
\begin{eqnarray}
&&\left. y_{t}-\Delta y-\frac{\lambda y}{|x|^{2}}-a(x)y=B_{1}w,\right.
\left. \mbox{in }(0,\infty )\times \Omega ,\right.  \label{70} \\
&&\left. y(t,0)=0,\mbox{ }y(t,1)=u\right. \left. \mbox{ \ \ \ \ \ \ \ \ \ \
\ for }t\geq 0,\right.  \label{71} \\
&&\left. y(0)=y_{0},\right. \left. \mbox{ \ \ \ \ \ \ \ \ \ \ \ \ \ \ \ \ \
\ \ \ \ \ \ \ \ \ in }\Omega ,\right.  \label{72} \\
&&\left. z=C_{1}y+D_{1}u,\right. \left. \mbox{ \ \ \ \ \ \ \ \ \ \ \ \ \ \ \
\ \ \ \ in }(0,\infty )\times \Omega ,\right.  \label{73}
\end{eqnarray}%
where $y_{0}\in L^{2}(\Omega ),$ $u\in \mathbb{R}.$

$\medskip $

$(i_{1})$ For this problem we choose $H=W=Z=L^{2}(\Omega ),$ $U=\mathbb{R},$%
\begin{equation}
B_{1}w=\chi _{\omega _{1}}(x)w,\mbox{ }C_{1}y=\chi _{\Omega _{C}}(x)y,\mbox{
\ }D_{1}u=d(x)u,\mbox{ }x\in \Omega ,  \label{73-2}
\end{equation}%
with the conditions $\omega _{1}\sqsubseteq \Omega ,$ $\Omega _{0}\subset
\Omega _{C},$ and 
\begin{equation}
d\in L^{2}(\Omega ),\mbox{ }d(x)=0\mbox{ on }\Omega _{C},\mbox{ }%
\int_{\Omega \backslash \Omega _{C}}d^{2}(x)dx=1.  \label{73-4}
\end{equation}%
Thus, $B_{1}\in L(L^{2}(\Omega ),L^{2}(\Omega )),$ $C_{1}\in L(L^{2}(\Omega
),L^{2}(\Omega ))$ and $D_{1}:U\rightarrow L^{2}(\Omega ).$

We deal again with the operator $A:D(A)\subset L^{2}(\Omega )\rightarrow
L^{2}(\Omega ),$ $Ay=\Delta y+\frac{\lambda y}{|x|^{2}},$ which is $\omega $-%
$m$-accretive on $L^{2}(\Omega )$ and generates a compact $C_{0}$-semigroup
on $L^{2}(\Omega ).$ The difference here is that in the calculus of the
accretivity of $-A$ we use the Hardy inequality (\ref{HN0}) instead of (\ref%
{HN}). Next, we define 
\begin{equation}
B:\mathbb{R\rightarrow R\times R}\mbox{, }Bu=(0,u)  \label{74}
\end{equation}%
and consider problem $\Delta \theta =0,$ $\theta =Bu$ on $\Gamma =\{0,1\}$
which provides the Dirichlet map $D_{0}u,$ associated to $\Delta $ and $Bu,$
expressed in this case by 
\begin{equation}
D_{0}u=ux.  \label{76}
\end{equation}%
Next, the problem%
\begin{equation}
\Delta Du+\frac{\lambda Du}{|x|^{2}}=0,\mbox{ }Du=Bu\mbox{ on }\Gamma ,
\label{77}
\end{equation}%
provides the Dirichlet map associated to $A_{0}$ defined in (\ref{47-2}).
Making the difference $\varphi =Du-D_{0}u$ we write the equation%
\begin{equation*}
\Delta \varphi +\frac{\lambda \varphi }{\left\vert x\right\vert ^{2}}=-\frac{%
\lambda u}{x},\mbox{ }\varphi =0\mbox{ on }\Gamma .
\end{equation*}%
By a similar calculus as in Lemma 5.1, where we note that in this case while
solving (\ref{51}) we have 
\begin{equation*}
\left( \frac{1}{2}-\frac{\lambda }{H_{N}}\right) \int_{\Omega }\left\vert
\nabla \varphi \right\vert ^{2}dx-\left\vert u\right\vert ^{2}\leq \Psi
(\varphi )<\infty ,
\end{equation*}%
we deduce that $\Psi $ has a minimum. Thus, we find that $\varphi \in
H_{0}^{1}(\Omega ),$ $\frac{\varphi }{x}\in L^{2}(\Omega )$ and 
\begin{equation}
Du=\varphi +ux\in H^{1}(\Omega ),\mbox{ }\frac{Du}{x}\in L^{2}(\Omega ).
\label{78}
\end{equation}

We define 
\begin{equation}
B_{2}:U=\mathbb{R}\rightarrow L^{2}(\Omega ),\mbox{ }B_{2}u=-\widetilde{A}%
Du+a(x)Du=-uA_{0}D(0,1)  \label{79}
\end{equation}%
where $\widetilde{A}$ is defined as in (\ref{54}) and $D(0,1)$ is the
Dirichlet map corresponding to the boundary data $y(t,0)=1,$ $y(t,1)=1.$
Then, $B_{2}^{\ast }:D(A)\rightarrow \mathbb{R}$ and Lemma 5.2 implies that 
\begin{equation}
B_{2}^{\ast }v=-v^{\prime }(1),\mbox{ }v\in D(A),\mbox{ }D^{\ast
}p=p^{\prime }(1),  \label{80}
\end{equation}%
where $D^{\ast }:L^{2}(\Omega )\rightarrow \mathbb{R}.$ We recall that $p$
is in $H^{2}$ in the neighborhood of the boundary $x=1.$

Hypotheses $(i_{2}),$ $(i_{3})$ and $(i_{4})$ are proved as in Section \ref%
{Boundary}.

Finally, we calculate the term $PB_{2}B_{2}^{\ast }P\varphi (x)$, the other
terms being the same as in the previous sections,%
\begin{equation*}
PB_{2}B_{2}^{\ast }P\varphi (x)=\int_{\Omega }\int_{\Omega }\frac{\partial
P_{0}}{\partial x}(1,\xi )\frac{\partial P_{0}}{\partial \xi }(x,1)\varphi
(\xi )d\xi
\end{equation*}%
and replacing in (\ref{15}) we get

\begin{theorem}
Let $\gamma >0$ and let $A,$ $B_{1},$ $C_{1}$ and $D_{1}$ be given by (\ref%
{45-1}) and (\ref{73-2}), respectively and $B_{2},$ $B_{2}^{\ast }$ be given
by (\ref{74}) and (\ref{80}). Assume that $P_{0}\in D(A)\times D(A)$ is a
solution to equation 
\begin{eqnarray}
&&\Delta _{x}P_{0}(x,\xi )+\Delta _{\xi }P_{0}(x,\xi )+\lambda P_{0}(x,\xi
)\left( \frac{1}{|x|^{2}}+\frac{1}{|\xi |^{2}}\right) +(a(x)+a(\xi
))P_{0}(x,\xi )  \notag \\
&&+\int_{\Omega }\frac{\partial P_{0}}{\partial x}(1,\xi )\frac{\partial
P_{0}}{\partial \xi }(x,1)d\xi +\gamma ^{-2}\int_{\Omega }\chi _{\omega
_{1}}(\overline{\xi })P_{0}(x,\overline{\xi })P_{0}(\overline{\xi },\xi )d%
\overline{\xi }  \label{81} \\
&=&-\delta (x-\xi )\chi _{\Omega _{C}}(\xi ),\mbox{ }(x,\xi )\in \Omega
\times \Omega ,  \notag
\end{eqnarray}%
with the boundary conditions $P_{0}(x,0)=P_{0}(x,1)=0$ for $x\in (0,1)$ and
by symmetry $P_{0}(0,\xi )=P_{0}(1,\xi )=0.$ Then, the feedback control $%
\widetilde{F}\in L(D(A),\mathbb{R}),$ 
\begin{equation}
\widetilde{F}y=\int_{\Omega }y(\xi )\frac{\partial P_{0}}{\partial x}(1,\xi
)d\xi ,\mbox{ }y\in L^{2}(\Omega )  \label{82}
\end{equation}%
solves the $H^{\infty }$-problem.
\end{theorem}

In this case 
\begin{equation*}
\Lambda _{P}y=A_{0}\left( y+\int_{\Omega }\frac{\partial P_{0}}{\partial x}%
(1,\xi )D(0,1)y(\xi )d\xi \right) +ay+\chi _{\omega _{1}}\int_{\Omega
}P_{0}(x,\xi )y(\xi )d\xi 
\end{equation*}%
which is closed, so that 
\begin{equation*}
\mathcal{X}=D(\Lambda _{P})=\left\{ y\in L^{2}(\Omega );\mbox{ }%
y+\int_{\Omega }\frac{\partial P_{0}}{\partial x}(1,\xi )D(0,1)y(\xi )d\xi
\in D(A)\right\} .
\end{equation*}

\bigskip

\noindent \textbf{Acknowledgment.} This work was supported by a grant of the Ministry of Research,
Innovation and Digitization, CNCS - UEFISCDI, project number
PN-III-P4-PCE-2021-0006, within PNCDI III.

\end{document}